\documentclass[english]{paper}
\usepackage[T1]{fontenc}
\usepackage[latin9]{inputenc}
\usepackage{mathrsfs}
\usepackage{amsmath}
\usepackage{amsthm}
\usepackage{amssymb}
\usepackage{wasysym}

\makeatletter
\numberwithin{equation}{section}
\numberwithin{figure}{section}
\theoremstyle{plain}
\newtheorem{thm}{\protect\theoremname}
\theoremstyle{definition}
\newtheorem{defn}[thm]{\protect\definitionname}
\theoremstyle{plain}
\newtheorem{lem}[thm]{\protect\lemmaname}
\theoremstyle{plain}
\newtheorem{prop}[thm]{\protect\propositionname}
\theoremstyle{plain}
\newtheorem{cor}[thm]{\protect\corollaryname}

\makeatother

\usepackage{babel}
\providecommand{\corollaryname}{Corollary}
\providecommand{\definitionname}{Definition}
\providecommand{\lemmaname}{Lemma}
\providecommand{\propositionname}{Proposition}
\providecommand{\theoremname}{Theorem}

\begin{document}
\title{A Constructive Brownian Limit Theorem }
\author{$\qquad$$\qquad$$\qquad$$\qquad$Yuen-Kwok Chan \thanks{unaffiliated}
\thanks{Contact: chan314@gmail.com}}

\maketitle
$\qquad$$\qquad$$\qquad$$\qquad$$\qquad$$\qquad$March 6, 2022
\begin{abstract}
In this paper, we present and prove a boundary limit theorem for Brownian
motions for the Hardy space $\mathbf{h}^{p}$ of harmonic functions
on the unit ball in $R^{m}$, where $p\geq1$ and $m\geq2$ are arbitrary.
Our proof is constructive in the sense of \cite{BishopBridges85,Chan21,Chan22}.
Roughly speaking, a mathematical proof is constructive if it can be
compiled into some computer code with the guarantee of exit in a finite
number of steps on execution. A constructive proof of said boundary
limit theorem is contained in \cite{Durret84} for the case of $p>1$.
In this article, we give a constructive proof for $p=1$, which then
implies, via the Lyapunov's inequality, a constructive proof for the
general case $p\geq1$. We conjecture that the result can be used
to give a constructive proof of the nontangential limit theorem for
Hardy spaces $\mathbf{h}^{p}$ with $p\geq1$.

We note that, $ca$ 1970, R. Getoor gave a talk on the Brownian limit
theorem at the University of Washington. We believe that the proof
he presented is constructive only for the case $p>1$ and not for
the case $p=1$. We are however unable to find a reference for his
proof.
\end{abstract}

\section{Preliminaries}

In this section we gather together basic notions in the integration
on an $(m-1)$-sphere, harmonic functions, Brownian motions, and exit
times. These can serve as a basis of future constructive development
in classical potential theory. We will also quote from \cite{Chan21}
Bishop's maximal inequality for martingales. The Brownian limit theorem,
the main theorem of this article, is then an easy application of said
inequality for martingales.

\subsection{Surface-area integration on an $(m-1)$-sphere}

Unless otherwise defined, notations and terminologies in this article
are from \cite{Chan21}; they conform mostly to familiar usage in
the probability literature. 

If $x,y$ are mathematical objects, we write $x\equiv y$ to mean
``$x$ is defined as $y$'', ``$x$, which is defined as $y$'',
``$x,$ which has been defined earlier as $y$'', or any other grammatical
variation depending on the context. 

Following \cite{BishopBridges85}, we define the functions $\cos,\sin:R\rightarrow[-1,1]$
in terms of the familiar power series, and define the constant $\pi$
as twice the smallest positive zero of the function $\cos.$ Without
further mention, we will use properties, proved in typical analysis
text books, of the functions $\cos,\sin$ and the constant $\pi$,
and of related functions. 
\begin{defn}
\textbf{Matrix notations.} The transpose of a matrix $\alpha$ is
denoted by $\alpha^{T}$. Let $p,q\geq1$ be arbitrary integers. Each
point $z\equiv(z_{1},\cdots,z_{p})\in R^{p}$ will be identified with
a column vector $\left[\begin{array}{c}
z_{1}\\
\vdots\\
z_{p}
\end{array}\right]$. In other words, it is regarded as a $p\times1$ matrix. Let $\left\Vert \cdot\right\Vert $
denote the Euclidean norm on the space $M^{p\times q}$ of $p\times q$
matrices, defined by 
\[
\left\Vert \alpha\right\Vert \equiv\sqrt{\sum_{i=1}^{p}\sum_{j=1}^{q}a_{i,j}^{2}}
\]
for each $\alpha\equiv[a_{i,j}]_{i=1,\cdots,p;j=1,\cdots,q}\in M^{p\times q}$.
In particular, $\left\Vert z\right\Vert \equiv\sqrt{z_{1}^{2}+\cdots+z_{p}^{2}}$
for each $z\equiv(z_{1},\cdots,z_{p})\in R^{p}.$ Thus $M^{p\times q}$
is equipped with the Euclidean metric $d_{ecld}$ defined by $d_{ecld}(\alpha,\beta)\equiv\left\Vert \alpha-\beta\right\Vert $
for each $\alpha,\beta\in M^{p\times q}$. At some small risk of confusion,
we suppress the reference to $p$ and $q$ in the definitions of $\left\Vert \cdot\right\Vert $
and of $d_{ecld}$, and will write $0$ for both the real number $0\in R$
and the matrix $0\in M^{p\times q}$ whose entries are all zeros.
Likewise, we will write $I$ for the identity matrix in $M^{p\times p}$.
Suppose, for each $j=1,\cdots,q,$ $\widehat{e}_{j}\equiv[\widehat{e}_{i,j}]_{i=1,\cdots,p}$
is a member of $M^{p\times1}.$ Then we write
\[
[\widehat{e}_{1},\cdots,\widehat{e}_{q}]\equiv[\widehat{e}_{i,j}]_{i=1,\cdots,p;j=1,\cdots,q}
\]
for the matrix in $M^{p\times q}$ whose $j$-th column is equal to
the vector $\widehat{e}_{j}$.

For arbitrary $z,x\in R^{p}$, the $1\times1$ matrix $z^{T}x$ is
identified with its sole entry, denoted by $z\cdot x$ and called
the inner product of the two vectors $z$ and $x$.
\end{defn}

$\square$
\begin{defn}
\textbf{Integration notations.} If $(S,L,\nu)$ is some complete integration
space in the sense of \cite{BishopBridges85}, and if $g\in L$ is
an arbitrary integrable function, then we write $\nu(g)$, $\nu g$,
$\int g(x)\nu(dx)$, and $\int\nu(dx)g(x)$ interchangeably. If, in
addition, $A$ is an integrable subset, then we write $\nu(g1_{A})$,
$\nu(g;A)$, $\int g(x)1_{A}(x)\nu(dx),$ $\int_{x\in A}\nu(dx)g(x),$and
$\int_{x\in A}\nu(dx)g(x)$ interchangeably, and write $A\in L$.
Moreover, if $S\equiv R^{m}$ and $\nu$ is the Lebesgue integration,
then we write $dx$ for $\nu(dx)$, and write 
\[
\int\cdots\int_{(x(1),\cdots,x(m))\in A}g(x_{1},\cdots,x_{m})dx_{1}\cdots dx_{m}
\]
\[
\equiv\int\cdots\int_{x\in A}g(x)dx\equiv\int_{x\in A}g(x)\nu(dx).
\]
If, in addition, $m=1$ and $A$ is a finite interval with end points
$a\leq b$, then we write $\int_{x=a}^{b}g(x)dx$ for $\int1_{A}(x)g(x)dx$.

Suppose $S\in L$ with $\nu(S)=1$. We will then call $(S,L,\nu)$
a probability space. When the sample space $S$ and the probability
integration $\nu$ are understood, we will also call $L$ a probability
space. We will write $\left\Vert g\right\Vert _{L}\equiv\nu|g|$ for
the $L_{1}$-norm of each $g\in L$.
\end{defn}

$\square$
\begin{defn}
\textbf{Miscellaneous notations and conventions.} To lessen the burden
on subscripts, for arbitrary expressions $a$ and $b$ we will write
the expressions $a_{b}$ and $a(b)$ interchangeably. Moreover, if
in a discussion involving a r.r.v. $X,$ the measurability of the
set $(X\leq t)$ is required, it will be assumed that $t$ has been
so chosen to ensure such measurability; similar assumptions when $\leq$
is replaced by $<$, $\geq$, or $>$.

For arbitrary real valued expressions $a$ and $b$, we will write
$a\vee b$ and $a\wedge b$ for $\max(a,b)$ and $\min(a,b)$ respectively.
Note that if $a$, $b$ and $c$ are real valued expressions with
$a\leq c$, then $(a\vee b)\wedge c=$ $a\vee(b\wedge c)$. We will
then omit the parentheses and write 
\[
a\vee b\wedge c\equiv(a\vee b)\wedge c=a\vee(b\wedge c).
\]
Moreover, at some small risk of confusion, we sometimes write $a=b\pm c$
for the inequality $|a-b|\leq c$.

We will let $[\cdot]_{1}$ denote the operation which assigns to each
$a\in R$ an integer $[a]_{1}$ which is in the interval $(a,a+2)$.
Note that there is no constructive proof that the operation $[\cdot]_{1}$
is a function. In other words, there is no guarantee that $[a]_{1}=[b]_{1}$
if $a=b$.
\end{defn}

$\square$
\begin{defn}
\textbf{Angle.} Let $m\geq2$ be arbitrary. Let $x,y\in R^{m}$ be
arbitrary such that $\left\Vert x\right\Vert >0$ and $\left\Vert y\right\Vert >0$.
We will let $\measuredangle(y,x)$ denote the unsigned \emph{angle
between the vectors} $x$ and $y$, defined as that unique real number
$\measuredangle(y,x)\equiv\theta\in[0,\pi]$ such that 
\[
\cos\theta=\frac{x\cdot y}{\left\Vert x\right\Vert \cdot\left\Vert y\right\Vert }.
\]
In other words, 
\[
\measuredangle(y,x)\equiv\arccos(\frac{x\cdot y}{\left\Vert x\right\Vert \cdot\left\Vert y\right\Vert }).
\]
\end{defn}

\begin{lem}
\textbf{\emph{\label{Lem. Continuity of cos(angle(x,y)) }Continuity
of }}\textbf{$\cos\measuredangle(y,x)$.} Let $r>0.$ be arbitrary.
Then $\cos\measuredangle(y,x)$ and $\sin\measuredangle(y,x)$ are
uniformly continuous functions on 
\[
G_{r}\equiv\{(x,y)\in R^{m}\times R^{m}:\left\Vert x\right\Vert \geq r,\left\Vert y\right\Vert \geq r\}.
\]
for each $r>0$. 
\end{lem}

\begin{proof}
Consider each $(x,y),(u,v)\in G_{r}$. Then 
\[
|\cos\measuredangle(y,x)-\cos\measuredangle(u,v)|=|\frac{x\cdot y}{\left\Vert x\right\Vert \cdot\left\Vert y\right\Vert }-\frac{u\cdot v}{\left\Vert u\right\Vert \cdot\left\Vert v\right\Vert }\cdot|
\]
\[
=|\frac{x\cdot y-x\cdot v}{\left\Vert x\right\Vert \cdot\left\Vert y\right\Vert }-\frac{x\cdot v-u\cdot v}{\left\Vert u\right\Vert \cdot\left\Vert v\right\Vert }|\leq|\frac{x\cdot y-x\cdot v}{\left\Vert x\right\Vert \cdot\left\Vert y\right\Vert }|+|\frac{x\cdot v-u\cdot v}{\left\Vert u\right\Vert \cdot\left\Vert v\right\Vert }|
\]
\[
\leq\frac{\left\Vert y-v\right\Vert }{\left\Vert y\right\Vert }+\frac{\left\Vert x-u\right\Vert }{\left\Vert u\right\Vert }
\]
\[
\leq r^{-1}(\left\Vert y-v\right\Vert +\left\Vert x-u\right\Vert ).
\]
Thus the function $\cos\measuredangle(y,x)$ is uniformly continuous
on $G_{r}$. It follows that the function
\[
\sin\measuredangle(y,x)=\sqrt{1-\cos^{2}\measuredangle(y,x)}
\]
is also uniformly continuous on $G_{r}$. 
\end{proof}
$\square$
\begin{defn}
\textbf{\label{Def. Rotation matrix and rotation mapping}Rotation
matrix and rotation mapping.} Let $m\geq2$ be arbitrary. An $m\times m$
matrix $\alpha$ is called an orthogonal matrix if $\alpha^{T}\alpha=I$,
where $I$ is the identity matrix. Let $O(m)$ denote the group of
orthogonal matrices, with matrix multiplication as the group multiplication.
$O(m)$ is called the \emph{orthogonal group}. A member of $O(m)$
is called a rotation matrix if it has determinant $+1$. The subgroup
$SO(m)$ of $O(m)$ consisting of rotation matrices is called the
\emph{rotation group.}

Let $y\in R^{m}$ and $\alpha\in SO(m)$ be arbitrary. Define the
mapping $\rho_{\alpha,y}:R^{m}\rightarrow R^{m}$ by 
\[
\rho_{\alpha,y}(z)\equiv y+\alpha\cdot(z-y)
\]
for each $z\in R^{m}$. Then $\rho_{\alpha,y}$ will be called the
\emph{rotation} of $R^{m}$ about the point  $y\in R^{m}$ by the
rotation matrix $\alpha\in SO(m)$. Note that $\rho_{\alpha,y}$ preserves
distance, with $\rho_{\alpha,y}(y)=y$. Moreover, $\rho_{\alpha,y}\circ\rho_{\alpha^{-1},y}$
is the identity mapping. In other words $\rho_{\alpha^{-1},y}$ is
the inverse function of $\rho_{\alpha,y}$.
\end{defn}

$\square$
\begin{defn}
\textbf{\label{Def. Ball and sphere} Ball and sphere.} Let $m\geq1$,
$x\in R^{m}$, and $r>s>0$ be arbitrary. Then the set 
\[
D_{x,r}^{m}\equiv\{y\in R^{m}:\left\Vert y-x\right\Vert <r\}
\]
is called the \emph{open $m$-ball} with center $x$ and radius $r$.
The set 
\[
\overline{D}_{x,r}^{m}\equiv\{y\in R^{m}:\left\Vert y-x\right\Vert \leq r\}
\]
is called the \emph{closed} \emph{$m$-ball} with center $x$ and
radius $r$. The set 
\[
D_{x,s,r}^{m}\equiv\{y\in R^{m}:s<\left\Vert y-x\right\Vert <r\}
\]
is called the \emph{open} \emph{$m$-shell} with center $x$, radius
$r$, and thickness $r-s$. The set 
\[
\overline{D}_{x,s,r}^{m}\equiv\{y\in R^{m}:s\leq\left\Vert y-x\right\Vert \leq r\}
\]
is called the \emph{closed} \emph{$m$-shell} with center $x$, radius
$r$, and thickness $r-s$. The set 
\[
\partial D_{x,r}^{m}\equiv\{y\in R^{m}:\left\Vert y-x\right\Vert =r\}
\]
is called the $(m-1)$-sphere with center $x$ and radius $r$. 

Each of the above-defined sets inherits the Euclidean metric $d_{ecld}$
from $R^{m}$. Let $C(\overline{D}_{x,r}^{m})$, $C(\overline{D}_{x,s,r}^{m})$,
and $C(\partial D_{x,r}^{m})$ denote the spaces of uniformly continuous
functions on the compact metric spaces $(\overline{D}_{x,r}^{m},d_{ecld})$,
$(\overline{D}_{x,s,r}^{m},d_{ecld})$, and $(\partial D_{x,r}^{m},d_{ecld})$
respectively. 

When the dimension $m$ is understood, and abbreviation desired, we
omit it in the subscripts and write $D_{x,r},$$\overline{D}_{x,r}$,
$D_{x,s,r}$, $\overline{D}_{x,s,r}$, $\partial D_{x,r}$ for $D_{x,r}^{m},$$\overline{D}_{x,r}^{m}$,
$D_{x,s,r}^{m}$, $\overline{D}_{x,s,r}^{m}$, $\partial D_{x,r}^{m}$
respectively.
\end{defn}

$\square$
\begin{defn}
\textbf{\label{Def. Spherical coordinates} Spherical coordinates.
}Let $m\geq2$ be arbitrary. Define the open upper half space
\begin{equation}
R_{+}^{m}\equiv\{x\equiv(x_{1},\cdots,x_{m})\in R^{m}:x_{m}>0\}\label{eq:temp-1}
\end{equation}
and open lower half space
\begin{equation}
R_{-}^{m}\equiv\{x\equiv(x_{1},\cdots,x_{m})\in R^{m}:x_{m}<0\}.\label{eq:temp-1-1}
\end{equation}
For each $x\equiv(x_{1},\cdots,x_{m})\in R_{+}^{m}\cup R_{-}^{m}$
and $i=1,\cdots,m-1$, define
\[
r(x)\equiv\left\Vert x\right\Vert >0,
\]
\[
\theta_{i}(x)\equiv r(x)^{-1}x_{i},
\]
Define the functions 
\[
\varphi_{+}(x)\equiv(r(x),\theta_{1}(x),\cdots,\theta_{m-1}(x))\in(0,\infty)\times D_{0,1}^{m-1}
\]
for each $x\equiv(x_{1},\cdots,x_{m})\in R_{+}^{m}$ and 
\[
\varphi_{-}(x)\equiv(r(x),\theta_{1}(x),\cdots,\theta_{m-1}(x))\in(0,\infty)\times D_{0,1}^{m-1}
\]
for each $x\equiv(x_{1},\cdots,x_{m})\in R_{-}^{m}$ Then the function
\[
\varphi_{+}:R_{+}^{m}\rightarrow(0,\infty)\times D_{0,1}^{m-1}
\]
is an injection. The $m$-tuple $\varphi_{+}(x)$ is called the spherical
coordinates of $x\in R_{+}^{m}$. Similarly, the function 
\[
\varphi_{-}:R_{-}^{m}\rightarrow(0,\infty)\times D_{0,1}^{m-1}
\]
is a injection. The $m$-tuple $\varphi_{-}(x)$ is called the spherical
coordinates of $x\in R_{-}^{m}$. 

Let 
\[
\psi_{+}:(0,\infty)\times D_{0,1}^{m-1}\rightarrow R_{+}^{m}
\]
and
\[
\psi_{-}:(0,\infty)\times D_{0,1}^{m-1}\rightarrow R_{-}^{m}
\]
denote the inverses of $\varphi_{+}$ and $\varphi_{-}$ respectively. 
\end{defn}

$\square$
\begin{prop}
\emph{\label{Prop:-Surface-area-integration on sphere centered at 0}}\textbf{\emph{
Surface-area integration on }}$(m-1)$-\textbf{\emph{sphere centered
at $0$. }}Let $m\geq2$, $y\in R^{m},$ and $r>0$ be arbitrary.
Let $g\in C(\partial D_{0,r})$ be arbitrary. Define
\[
g_{+}(x_{1},\cdots,x_{m})\equiv g(x_{1},\cdots,x_{m})1_{x(m)>0}
\]
and
\[
g_{-}(x_{1},\cdots,x_{m})\equiv g(x_{1},\cdots,x_{m})1_{x(m)<0}
\]
for each $x\equiv(x_{1},\cdots,x_{m})\in\partial D_{0,r}$. Following
page 4 of \cite{Helms69}, define the integral 
\[
\sigma_{m,0,r}(g)\equiv\int\sigma_{m,0,r}(dz)g(z)
\]
\[
\equiv\int\cdots\int_{(\theta_{1},\cdots,\theta_{m-1})\in D^{m-1}(0,1)}d\theta_{1}\cdots d\theta_{m-1}\frac{r^{m-1}}{\sqrt{1-\theta_{1}^{2}-\cdots-\theta_{m-1}^{2}}}
\]
\[
\{g(r\theta_{1},\cdots,r\theta_{m-1},+r\sqrt{1-\theta_{1}^{2}-\cdots-\theta_{m-1}^{2}})
\]
\begin{equation}
+g(r\theta_{1},\cdots,r\theta_{m-1},-r\sqrt{1-\theta_{1}^{2}-\cdots-\theta_{m-1}^{2}})\}\label{eq:temp-51}
\end{equation}
Then the following conditions hold.

1. $\int\sigma_{m,0,r}(dz)g(z)=r^{m-1}\int\sigma_{m,0,1}(dz)g(rz)$.

2. The function $\sigma_{m,0,r}:C(\partial D_{0,r})\rightarrow R$
is an integration on $\partial D_{0,r}$. It will be called the \emph{surface-area
integration} on the $(m-1)$-sphere $\partial D_{0,r}$. 

3. Let $s\in(0,r)$ be arbitrary. Define the constants 
\[
\sigma_{m,r}\equiv\sigma_{m,0,r}(1),
\]
\[
\nu_{m,r}\equiv\int\cdots\int_{x\in D(0,r)}1dx
\]
 and 
\[
\nu_{m,s,r}\equiv\int_{x\in\overline{D}(0,s,r)}1dx=\int_{x\in D(0,s,r)}1dx=\nu_{m,r}-\nu_{m,s}.
\]
Then we have
\begin{equation}
\sigma_{m,r}^{-1}\int_{z\in\partial D(0,r)}\sigma_{m,0,r}(dz)g(z)=\lim_{s\uparrow r}\nu_{m,s,r}^{-1}\int\cdots\int_{x\in D(0,s,r)}g(\frac{rx}{\left\Vert x\right\Vert })dx\label{eq:temp-12}
\end{equation}
for each $g\in C(\partial D_{0,r})$. Thus the surface distribution
$\sigma_{m,r}^{-1}\sigma_{m,0,r}$ can be defined in terms of Lebesgue
integration.

4. Let $y\in R^{m}$ and $\alpha\in SO(m)$ be arbitrary. Define the
\emph{surface area integration} $\sigma_{m,y,r}$ on $\partial D_{y,r}$
by 
\begin{equation}
\sigma_{m,y,r}(g)\equiv\int\sigma_{m,0,r}(dz)g(z-y)\label{eq:temp-49}
\end{equation}
for each $g\in C(\partial D_{y,r})$. Then a function $g$ on $\partial D_{y,r}$
is integrable relative to $\sigma_{m,y,r}$ iff the function $g\circ\rho_{\alpha,y}$
is integrable relative to $\sigma_{m,y,r}$, in which case 
\[
\sigma_{m,y,r}(g)=\sigma_{m,y,r}(g\circ\rho_{\alpha,y}).
\]
In short, the surface-area integration $\sigma_{m,y,r}$ on the $(m-1)$-sphere
$\partial D_{y,r}$ is invariant relative to rotations about $y$.
\end{prop}

\begin{proof}
1. Assertion 1 can be verified directly from the defining equality
\ref{eq:temp-51}.

2. From the defining equality \ref{eq:temp-51}, we see that the function
$\sigma_{m,0,r}$ is linear in $g\in C(\partial D_{0,r})$. Now suppose
$g\in C(\partial D_{0,r})$ is such that $\sigma_{m,0,r}(g)>0$. Then,
because the Lebesgue integration
\[
\int\cdots\int_{(\theta_{1},\cdots,\theta_{m-1})\in D^{m-1}(0,1)}\cdot d\theta_{1}\cdots d\theta_{m-1}
\]
is an integration, the expression in the braces on the right-hand
side of equality \ref{eq:temp-51} is positive at some $(\theta_{1},\cdots,\theta_{m-1})\in D_{m-1,0,1}$.
Consequently, $g(x)>0$ for some
\[
x\equiv(r\theta_{1},\cdots,r\theta_{m-1},+r\sqrt{1-\theta_{1}^{2}-\cdots-\theta_{m-1}^{2}})
\]
or 
\[
x\equiv(r\theta_{1},\cdots,r\theta_{m-1},-r\sqrt{1-\theta_{1}^{2}-\cdots-\theta_{m-1}^{2}}).
\]
Thus we have verified the positivity condition, condition (ii), in
definition 4.2.1 in \cite{Chan21} for the linear function $\sigma_{m,0,r}$
to be an integration on the compact metric space $\partial D_{0,r}$.
Assertion 2 is proved.

3. To prove Assertion 3, let $s\in(0,r)$ be arbitrary. Define the
function $\overline{g}\in C(\overline{D}_{0,s,r})$ by $\overline{g}(x)\equiv g(\frac{rx}{\left\Vert x\right\Vert })$
for each $\overline{D}_{0,s,r}$. 

4. Consider the limit
\[
\lim_{s\uparrow r}\nu_{m,s,r}^{-1}\int\cdots\int_{x\in D(0,s,r)}g(\frac{rx}{\left\Vert x\right\Vert })dx
\]
\[
=\lim_{s\uparrow r}\nu_{m,s,r}^{-1}\int\cdots\int_{x\in R_{+}^{m}D(0,s,r)}g(\frac{rx}{\left\Vert x\right\Vert })dx
\]
\begin{equation}
+\lim_{s\uparrow r}\nu_{m,s,r}^{-1}\int\cdots\int_{x\in R_{-}^{m}D(0,s,r)}g(\frac{rx}{\left\Vert x\right\Vert })dx.\label{eq:temp-8}
\end{equation}
Apply a change of integration-variables to the first integral on the
right-hand side, with
\[
x\equiv(x_{1},\cdots,x_{m})=(t\theta_{1},\cdots,t\theta_{m-1},+t\sqrt{1-\theta_{1}^{2}-\cdots-\theta_{m-1}^{2}}).
\]
The Jacobian is 
\[
|\det\left[\begin{array}{ccccc}
\theta_{1}, & \theta_{2}, & \cdots, & \theta_{m-1}, & (1-\theta_{1}^{2}-\cdots-\theta_{m-1}^{2})^{1/2}\\
t, & 0, & \cdots, & 0, & -t\theta_{1}(1-\theta_{1}^{2}-\cdots-\theta_{m-1}^{2})^{-1/2}\\
0, & t, & \cdots, & 0, & -t\theta_{2}(1-\theta_{1}^{2}-\cdots-\theta_{m-1}^{2})^{-1/2}\\
\vdots & \vdots & \ddots & \vdots & \vdots\\
0, & 0, & \cdots, & t & -t\theta_{m-1}(1-\theta_{1}^{2}-\cdots-\theta_{m-1}^{2})^{-1/2}
\end{array}\right]|
\]
\[
=|t^{m-1}(1-\theta_{1}^{2}-\cdots-\theta_{m-1}^{2})^{1/2}+t^{m-1}\theta_{1}^{2}(1-\theta_{1}^{2}-\cdots-\theta_{m-1}^{2})^{-1/2}
\]
\[
+\cdots+t^{m-1}\theta_{m-1}^{2}(1-\theta_{1}^{2}-\cdots-\theta_{m-1}^{2})^{-1/2}|
\]
\[
=|t^{m-1}(1-\theta_{1}^{2}-\cdots-\theta_{m-1}^{2})^{-1/2}|\cdot|(1-\theta_{1}^{2}-\cdots-\theta_{m-1}^{2})+\theta_{1}^{2}
\]
\[
+\theta_{1}^{2}+\cdots+\theta_{m-1}^{2}|
\]
\[
=|t^{m-1}(1-\theta_{1}^{2}-\cdots-\theta_{m-1}^{2})^{-1/2}|
\]
The first limit on the right-hand side of equality \ref{eq:temp-8}
then becomes
\[
\lim_{s\uparrow r}\nu_{m,s,r}^{-1}\int\cdots\int_{x\in R_{+}^{m}D(0,s,r)}g(\frac{rx}{\left\Vert x\right\Vert })dx
\]
\[
=\lim_{s\uparrow r}\nu_{m,s,r}^{-1}\int_{t=s}^{r}\int\cdots\int_{(\theta_{1},\cdots,\theta_{m-1})\in D^{m-1}(0,1)}dtd\theta_{1}\cdots d\theta_{m-1}\frac{t^{m-1}}{\sqrt{1-\theta_{1}^{2}-\cdots-\theta_{m-1}^{2}}}
\]
\[
g(r\theta_{1},\cdots,r\theta_{m-1},+r\sqrt{1-\theta_{1}^{2}-\cdots-\theta_{m-1}^{2}})
\]
\[
=\lim_{s\uparrow r}(r^{m}\nu_{m,0,1}-s^{m}\nu_{m,0,1})^{-1}m^{-1}(r^{m}-s^{m})
\]
\[
\int\cdots\int_{(\theta_{1},\cdots,\theta_{m-1})\in D^{m-1}(0,1)}d\theta_{1}\cdots d\theta_{m-1}\frac{1}{\sqrt{1-\theta_{1}^{2}-\cdots-\theta_{m-1}^{2}}}
\]
\[
g(r\theta_{1},\cdots,r\theta_{m-1},+r\sqrt{1-\theta_{1}^{2}-\cdots-\theta_{m-1}^{2}})
\]
\[
=\nu_{m,0,1}^{-1}m^{-1}\int\cdots\int_{(\theta_{1},\cdots,\theta_{m-1})\in D^{m-1}(0,1)}d\theta_{1}\cdots d\theta_{m-1}\frac{1}{\sqrt{1-\theta_{1}^{2}-\cdots-\theta_{m-1}^{2}}}
\]
\[
g(r\theta_{1},\cdots,r\theta_{m-1},+r\sqrt{1-\theta_{1}^{2}-\cdots-\theta_{m-1}^{2}}).
\]
Similarly, the second limit on the right-hand side of equality \ref{eq:temp-8}
is equal to
\[
\nu_{m,0,1}^{-1}m^{-1}\int\cdots\int_{(\theta_{1},\cdots,\theta_{m-1})\in D^{m-1}(0,1)}d\theta_{1}\cdots d\theta_{m-1}\frac{1}{\sqrt{1-\theta_{1}^{2}-\cdots-\theta_{m-1}^{2}}}
\]
\[
g(r\theta_{1},\cdots,r\theta_{m-1},-r\sqrt{1-\theta_{1}^{2}-\cdots-\theta_{m-1}^{2}}).
\]
Combining, equality \ref{eq:temp-8} yields
\[
\lim_{s\uparrow r}\nu_{m,s,r}^{-1}\int\cdots\int_{x\in D(0,s,r)}g(\frac{rx}{\left\Vert x\right\Vert })dx
\]
\[
=\nu_{m,0,1}^{-1}m^{-1}\int\cdots\int_{(\theta_{1},\cdots,\theta_{m-1})\in D^{m-1}(0,1)}\frac{1}{\sqrt{1-\theta_{1}^{2}-\cdots-\theta_{m-1}^{2}}}
\]
\[
\{g(r\theta_{1},\cdots,r\theta_{m-1},+r\sqrt{1-\theta_{1}^{2}-\cdots-\theta_{m-1}^{2}})+g(r\theta_{1},\cdots,r\theta_{m-1},-r\sqrt{1-\theta_{1}^{2}-\cdots-\theta_{m-1}^{2}})\}
\]
\[
=\nu_{m,0,1}^{-1}m^{-1}r^{-m+1}\int\cdots\int_{(\theta_{1},\cdots,\theta_{m-1})\in D^{m-1}(0,1)}\frac{r^{m-1}}{\sqrt{1-\theta_{1}^{2}-\cdots-\theta_{m-1}^{2}}}
\]
\[
\{g(r\theta_{1},\cdots,r\theta_{m-1},+r\sqrt{1-\theta_{1}^{2}-\cdots-\theta_{m-1}^{2}})
\]
\[
+g(r\theta_{1},\cdots,r\theta_{m-1},-r\sqrt{1-\theta_{1}^{2}-\cdots-\theta_{m-1}^{2}})\}
\]
\begin{equation}
=\nu_{m,0,1}^{-1}m^{-1}r^{-m+1}\int\sigma_{m,0,r}(dz)g(z),\label{eq:temp-9}
\end{equation}
where the last equality is thanks to the defining equality \ref{eq:temp-51}.

For the special case where $g\equiv1$, equality \ref{eq:temp-9}
reduces to
\[
1=\nu_{m,0,1}^{-1}m^{-1}r^{-m+1}\sigma_{m,r}.
\]
Hence equality \ref{eq:temp-9} can be rewritten as
\[
\lim_{s\uparrow r}\nu_{m,s,r}^{-1}\int\cdots\int_{x\in D(0,s,r)}g(\frac{rx}{\left\Vert x\right\Vert })dx=\sigma_{m,r}^{-1}\int_{z\in\partial D(0,r)}\sigma_{m,0,r}(dz)g(z),
\]
 as alleged in equality \ref{eq:temp-12} of Assertion 3. 

4. Next, let $\alpha$ be an arbitrary rotation matrix. Consider each
$g\in C(\partial D_{0,r})$. Then equality \ref{eq:temp-12} implies
that 
\[
\sigma_{m,0,r}(g\circ\alpha\cdot)=\sigma_{m,r}\lim_{s\uparrow r}\nu_{m,s,r}^{-1}\int\cdots\int_{x\in D(0,s,r)}g(\frac{r\alpha x}{\left\Vert \alpha x\right\Vert })dx
\]
\[
=\sigma_{m,r}\lim_{s\uparrow r}\nu_{m,s,r}^{-1}\int\cdots\int_{x\in D(0,s,r)}g(\frac{r\alpha x}{\left\Vert x\right\Vert })dx
\]
\[
=\sigma_{m,r}\lim_{s\uparrow r}\nu_{m,s,r}^{-1}\int\cdots\int_{x\in D(0,s,r)}g(\alpha x)dx
\]
\[
=\sigma_{m,r}\lim_{s\uparrow r}\nu_{m,s,r}^{-1}\int\cdots\int_{z\in D(0,s,r)}g(z)\cdot|\det\alpha^{-1}|\cdot dz_{1}\cdots dz_{m}
\]
\[
=\sigma_{m,r}\lim_{s\uparrow r}\nu_{m,s,r}^{-1}\int\cdots\int_{z\in D(0,s,r)}g(z)dz_{1}\cdots dz_{m}
\]
\begin{equation}
=\sigma_{m,0,r}(g).\label{eq:temp-7}
\end{equation}

5. Next let $y\in R^{m}$ be arbitrary. Define the integration $\sigma_{m,y,r}$
on $\partial D_{y,r}$ by 
\begin{equation}
\sigma_{m,y,r}(g)\equiv\sigma_{m,0,r}(g(\cdot-y))\label{eq:temp-46}
\end{equation}
for each $g\in C(\partial D_{y,r})$. 

6. Now let $y\in R^{m}$ and $\alpha\in SO(m)$ be arbitrary. Consider
each $g\in C(\partial D_{y,r})$. We will prove that 
\begin{equation}
\sigma_{m,y,r}(g\circ\rho_{\alpha,y})=\sigma_{m,y,r}(g).\label{eq:temp-47}
\end{equation}
To that end, note that the left-hand side of equality \ref{eq:temp-47}
can be written as
\[
\sigma_{m,y,r}(g\circ\rho_{\alpha,y})\equiv\int_{z\in\partial D(y,r)}\sigma_{m,y,r}(dz)g(\rho_{\alpha,y}(z))
\]
\[
=\int_{z\in\partial D(y,r)}\sigma_{m,y,r}(dz)g(y+\alpha\cdot(z-y))
\]
\[
=\int_{u\in\partial D(0,r)}\sigma_{m,0,r}(du)g(y+\alpha u)
\]
\[
=\int_{u\in\partial D(0,r)}\sigma_{m,0,r}(du)g(y+u)
\]
\[
=\int_{v\in\partial D(y,r)}\sigma_{m,y,r}(dv)g(v)\equiv\sigma_{m,y,r}(g),
\]
where the fourth equality is by equality \ref{eq:temp-7}, and where
$g\in C(\partial D_{y,r})$ is arbitrary. Thus the integrations $\sigma_{m,y,r}(\cdot\circ\rho_{\alpha,y})$
and $\sigma_{m,y,r}$ on $\partial D_{y,r}$ are equal. Therefore
a function $g$ on $\partial D_{y,r}$ is integrable relative to $\sigma_{m,y,r}$
iff the function $g\circ\rho_{\alpha,y}$ is integrable relative to
$\sigma_{m,y,r}$, in which case $\sigma_{m,y,r}(g)=\sigma_{m,y,r}(g\circ\rho_{\alpha,y})$.
Assertion 4 and the proposition are proved.
\end{proof}
\begin{defn}
\textbf{Total surface-area of an $(m-1)$-sphere centered at}\textbf{\emph{
$0$.}} The integral $\sigma_{m,r}\equiv\sigma_{m,0,r}(1)$ is called
the \emph{total surface area} of the $(m-1)$-sphere $\partial D_{m,0,r}$.
The Lebesgue integral $\nu_{m,r}\equiv\int_{x\in D_{m,0,r}}1dx$ is
called the \emph{total volume} of the unit ball $D_{m,0,r}$. 
\end{defn}

\begin{prop}
\textbf{\emph{Total volume of $m$-ball and total surface area of
$(m-1)$-sphere.}} Let $m\geq2$ and $r>0$ be arbitrary. Then the
following conditions hold.

1. $\nu_{m,r}=r^{m}\nu_{m,1}.$ 

2. $\sigma_{m,r}=r^{m-1}\sigma_{m,1}$.

3. $\nu_{m,1}=\frac{1}{m}\sigma_{m,1}$.

4. 
\[
\sigma_{m,1}=\begin{cases}
\begin{array}{c}
\frac{\pi^{m/2}m}{(m/2)!}\qquad\qquad\qquad\qquad\qquad\mathrm{if\;\mathit{m}\;is\;even},\\
\frac{2^{(m+1)/2}\pi^{(m-1)/2}}{1\cdot3\cdot5\cdots(m-2)}\qquad\qquad\qquad\mathrm{if\;\mathit{m}\;is\;odd}.
\end{array}\end{cases}
\]
\end{prop}

\begin{proof}
See pages 3, 4, and 5 of \cite{Helms69}.
\end{proof}
\begin{defn}
\label{Def. Uniform distribttionn on (m-1)-sphere}\textbf{ Uniform
distribution on $(m-1)$-sphere. }The normalized surface area integration
$\overline{\sigma}_{m,y,r}\equiv\sigma_{m,r}^{-1}\sigma_{m,y,r}$
is called the \emph{uniform distribution} on the $(m-1)$-sphere $\partial D_{y,r}$.
Note that, since $\sigma_{m,y,r}$ is invariant relative to rotations
$\rho_{\alpha,y}$ about the center $y$, so is the uniform distribution
$\overline{\sigma}_{m,y,r}$. We will write $L(\overline{\sigma}_{m,y,r})$
for the probability space generated by the family $C(\partial D_{y,r})$
of continuous functions on $\partial D_{y,r}$ relative to the uniform
distribution $\overline{\sigma}_{m,y,r}$.
\end{defn}

\begin{thm}
\label{Thm. distribution on D(0,1) invariant relative to rotation is equal to uniform distribution}
\textbf{\emph{Each distribution on the $(m-1)$-sphere $\partial D_{0,1}$
that is invariant relative to rotations is equal to the uniform distribution.
}}Let $\sigma$ be an arbitrary distribution on $\partial D_{0,1}$that
is invariant relative to rotations. Then $\sigma=\overline{\sigma}_{m,0,1}$.
\end{thm}

\begin{proof}
For ease of notations, we will give the proof only for the case $m\geq3$.
The case where $m=2$ would be along similar lines and simpler. The
proof will be by means of Haar measures as in chapter 8 of \cite{BishopBridges85}.

1. Let the points $\widehat{x}_{1},\cdots,\widehat{x}_{m}\in R^{m}$
be the natural basis of $R^{m}$. In other words, $\left[\widehat{x}_{1},\cdots,\widehat{x}_{m}\right]$
is the identity $m\times m$ matrix $I$.

2. First consider the case where $y=0\in R^{m}$ and $r=1$. Fix the
reference point $\widehat{e}\equiv\widehat{e}_{1}\equiv\widehat{x}_{1}\equiv(1,0,\cdots,0)\in\partial D_{0,1}.$
For each $k\geq1$, define the set 
\[
G_{k}\equiv\{z\in\partial D_{0,1}:\left\Vert z-\widehat{e}\right\Vert \geq2^{-k}\}.
\]
Define the dense subsets 
\[
G\equiv\bigcup_{k=1}^{\infty}G_{k}\equiv\{z\in\partial D_{0,1}:\left\Vert z-\widehat{e}\right\Vert >0\}
\]
and 
\[
G_{\vee}\equiv G\cup\{\widehat{e}\}.
\]
of $(\partial D_{0,1},d_{ecld})$. 

3. Consider each $z\in\partial D_{0,1}$. Let $\theta_{z}\equiv\measuredangle(\widehat{e},z)\in[0,\pi]$.
Then 
\[
\cos\theta_{z}=\frac{z\cdot\widehat{e}}{\left\Vert z\right\Vert \cdot\left\Vert \widehat{e}\right\Vert }=z\cdot\widehat{e}.
\]
Thus $\cos\theta_{z}$ is a uniformly continuous function of $z\in\partial D_{0,1}$.
Hence $\sin\theta_{z}=\sqrt{1-\cos^{2}\theta_{z}}$ is also a uniformly
continuous function of $z\in\partial D_{0,1}$. 

4. Next, let $k\geq1$ and $z\in G_{k}$ be arbitrary. Write $i_{1}\equiv1$
and $i_{2}\equiv2$. Let $\widehat{e}_{1,z}\equiv\widehat{e}_{1}\equiv\widehat{e}$
and define 
\begin{equation}
\widehat{e}_{2}\equiv\widehat{e}_{2,z}\equiv\frac{z-(z\cdot\widehat{e})\widehat{e}}{\left\Vert z-(z\cdot\widehat{e})\widehat{e}\right\Vert }=\frac{z-(\cos\theta_{z})\widehat{e}}{\left\Vert z-(\cos\theta_{z})\widehat{e}\right\Vert }.\label{eq:temp-41}
\end{equation}
Then 
\[
\left\Vert z-(z\cdot\widehat{e})\widehat{e}\right\Vert ^{2}=1-(z\cdot\widehat{e})^{2}=2^{-1}(2-2(z\cdot\widehat{e})^{2})=2^{-1}\left\Vert z-\widehat{e}\right\Vert ^{2}\geq2^{-k-1}.
\]
Hence $\widehat{e}_{2,z}$ is well defined and is a uniformly continuous
function on $G_{k}$, where $k\geq1$ is arbitrary. Moreover, $\widehat{e}_{1}\cdot\widehat{e}_{2}=0$.
Furthermore, equality \ref{eq:temp-41} yields
\begin{equation}
(\cos\theta_{z})\widehat{e}_{1}+\left\Vert z-(\cos\theta_{z})\widehat{e}_{1}\right\Vert \widehat{e}_{2}=z.\label{eq:temp-42}
\end{equation}
Since $\widehat{e}_{1},\widehat{e}_{2},z$ are unit vectors, while
$\widehat{e}_{1},\widehat{e}_{2}$ are mutually orthogonal, it follows
that
\[
(\cos\theta_{z})^{2}+\left\Vert z-(\cos\theta_{z})\widehat{e}_{1}\right\Vert ^{2}=1,
\]
whence
\[
\left\Vert z-(\cos\theta_{z})\widehat{e}_{1}\right\Vert =\sqrt{1-(\cos\theta_{z})^{2}}=\sin\theta_{z}.
\]
Therefore equality \ref{eq:temp-42} can be restated as
\begin{equation}
(\cos\theta_{z})\widehat{e}_{1}+(\sin\theta_{z})\widehat{e}_{2}=z.\label{eq:temp-44}
\end{equation}

3. Let 
\[
V_{z}\equiv\{c_{1}\widehat{e}_{1}+c_{2}\widehat{e}_{2}:c_{2},c_{2}\in R\}
\]
be the $2$-dimensional subspace spanned by $\widehat{e}_{1}.\widehat{e}_{2}$.
Let .
\[
V_{z}^{\bot}\equiv\{u\in R^{m}:u\cdot v=0\quad\mathrm{for\;each}\;v\in V\}
\]
be the $(m-1)$-dimensional subspace that is orthogonal to $V_{z}$.
Let $\widehat{e}_{3},\cdots,\widehat{e}_{m}$ be an arbitrary orthonormal
basis of $V_{z}^{\bot}$. Then $\widehat{e}_{1},\cdots,\widehat{e}_{m}$
be an arbitrary orthonormal basis of $R^{m}$. 

4. Define the rotation matrix
\[
\alpha_{z}\equiv\left[\widehat{e}_{1},\widehat{e}_{2},\widehat{e}_{3},\cdots,\widehat{e}_{m}\right]\left[\begin{array}{cccccc}
\cos\theta_{z}, & -\sin\theta_{z}, & 0, & 0, & \cdots, & 0\\
\sin\theta_{z}, & \cos\theta_{z} & 0, & 0, & \cdots, & 0\\
0, & 0, & 1, & 0, & \cdots, & 0\\
0, & 0, & 0, & 1, & \cdots, & 0\\
\vdots & \vdots & \vdots & \vdots & \ddots & \vdots\\
0, & 0, & 0, & 0, & 0, & 1
\end{array}\right]\left[\begin{array}{c}
\widehat{e}_{1}^{T}\\
\widehat{e}_{2}^{T}\\
\widehat{e}_{3}^{T}\\
\widehat{e}_{4}^{T}\\
\vdots\\
\widehat{e}_{m}^{T}
\end{array}\right]
\]
\[
=\left[\widehat{e}_{1},\widehat{e}_{2},\widehat{e}_{3},\widehat{e}_{4},\cdots,\widehat{e}_{m}\right]\left[\begin{array}{c}
\widehat{e}_{1}^{T}\cos\theta_{z}-\widehat{e}_{2}^{T}\sin\theta_{z}\\
\widehat{e}_{1}^{T}\sin\theta_{z}+\widehat{e}_{2}^{T}\cos\theta_{z}\\
\widehat{e}_{3}^{T}\\
\widehat{e}_{4}^{T}\\
\vdots\\
\widehat{e}_{m}^{T}
\end{array}\right]
\]
\begin{equation}
=\widehat{e}_{1}(\widehat{e}_{1}^{T}\cos\theta_{z}-\widehat{e}_{2}^{T}\sin\theta_{z})+\widehat{e}_{2}(\widehat{e}_{1}^{T}\sin\theta_{z}+\widehat{e}_{2}^{T}\cos\theta_{z})+\sum_{i=3}^{m}\widehat{e}_{i}\widehat{e}_{i}^{T}.\label{eq:temp-26}
\end{equation}
Right multiplication by $\widehat{e}_{1}$ yields
\[
\alpha_{z}\widehat{e}_{1}=\widehat{e}_{1}\cos\theta_{z}+\widehat{e}_{2}\sin\theta_{z}=z.
\]
Right multiplication by $\widehat{e}_{2}$ yields
\[
\alpha_{z}\widehat{e}_{2}=-\widehat{e}_{1}\sin\theta_{z}+\widehat{e}_{2}\cos\theta_{z}
\]

5. We will show that the matrix $\alpha_{z}$ is independent of $\widehat{e}_{3},\cdots,\widehat{e}_{m}.$
To that end, let $\widetilde{e}_{3},\cdots,\widetilde{e}_{m}$ be
a second arbitrary orthonormal basis for $V^{\bot}$. Then there exists
an orthogonal matrix $\beta\in O(m)$ such that $\widehat{e}_{i}=\widetilde{e}_{i}\beta$
for each $i=1,\cdots,m$. Hence
\[
\sum_{i=3}^{m}\widehat{e}_{i}\widehat{e}_{i}^{T}=\sum_{i=3}^{m}\widetilde{e}_{i}\beta\beta^{T}\widehat{e}_{i}^{T}=\sum_{i=3}^{m}\widetilde{e}_{i}\widehat{e}_{i}^{T}.
\]
Equality \ref{eq:temp-26} therefore shows that the matrix $\alpha_{z}$
is independent of the orthonormal basis $\widehat{e}_{3},\cdots,\widehat{e}_{m}$
of the space $V^{\bot},$ and is completely determined by $z$, where
$z\in G_{k}$ and $k\geq1$ are arbitrary. 

6. Therefore the operation $\psi:G_{\vee}\rightarrow SO(m)$ defined
on 
\[
G_{\vee}\equiv\bigcup_{k=1}^{\infty}G_{k}\cup\{\widehat{e}\}\equiv G\cup\{\widehat{e}\}.
\]
by 
\[
\psi(z)\equiv\alpha_{z}
\]
for each $z\in G$ and by $\psi(\widehat{e})\equiv I$, is a well
defined function. Moreover, 
\[
\psi(z)\widehat{e}=z.
\]
for each $z\in G.$ 

7. We have already seen that the function $\psi$ is uniformly continuous
on each compact subset of $G$. We will next prove that $\psi$ is
uniformly continuous on $G_{\vee}$. To that end, it suffices to prove
that $\psi$ is continuous at $\widehat{e}$. Note that, as $z\rightarrow\widehat{e}$
with $z\in G$, we have $\cos\theta_{z}\rightarrow1$ and $\sin\theta_{z}\rightarrow0$.
Hence, equality \ref{eq:temp-26} shows that, as $z\rightarrow\widehat{e}$,
we have 
\[
\psi(z)\equiv\alpha_{z}\rightarrow\sum_{i=1}^{m}\widehat{e}_{i}\widehat{e}_{i}^{T}=I\equiv\psi(\widehat{e}).
\]
Thus $\psi$ is continuous at $\widehat{e}$, and therefore uniformly
continuous on $G_{\vee}$. 

8. Since $G_{\vee}$ is a dense subset of $\partial D_{0,1}$ the
function $\psi$ can be extended to a continuous function.
\[
\psi:\partial D_{0,1}\rightarrow SO(m)
\]
such that $\psi(z)\widehat{e}=z$ for each $z\in\partial D_{0,1}$.
Let 
\[
\overline{SO}(m)\equiv\psi(\partial D_{0,1})
\]
denote the range of $\psi$. Then 
\[
\psi:\partial D_{0,1}\rightarrow\overline{SO}(m)
\]
is a continuous surjection.

9. In the other direction, define the function 
\[
\varphi:\overline{SO}(m)\rightarrow\partial D_{0,1}
\]
by $\varphi(\alpha)\equiv\alpha\widehat{e}\in\partial D_{0,1}$ for
each $\alpha\in\overline{SO}(m)$. Then $\varphi$ is continuous and
\[
\varphi\psi(z)\equiv\psi(z)\widehat{e}=z
\]
for each $z\in\partial D_{0,1}$. Thus $\varphi$ is the inverse of
the continuous function $\psi$. Hence 
\[
\psi:(\partial D_{0,1},d_{ecld})\rightarrow(\overline{SO}(m),d_{ecld})
\]
is a homeomorphism. Therefore we can regard each point $z\in\partial D_{0,1}$
as a rotation matrix $\psi(z)\in\overline{SO}(m)$, and write $z\cdot w$
for the matrix product $\psi(z)\cdot\psi(w)$ for each $z,w\in\partial D_{0,1}$

10. Since $(\overline{SO}(m),d_{ecld},\cdot)$ is a compact group,
the $(m-1)$-sphere $\partial D_{0,1}$becomes a compact group, with
the matrix multiplication as the group operations, and with $\widehat{e}\cdot w\equiv\psi(\widehat{e})\cdot w=I\cdot w=w$
for each $w\in\partial D_{0,1}$. Thus the reference point $\widehat{e}$
is the identity element in the compact group $\partial D_{0,1}$.

11. Now suppose $\sigma$ is an arbitrary distribution on $\partial D_{0,1}$
which is invariant relative to rotations. Then, for each $z\in\partial D_{0,1}$,
we have 
\[
\int_{w\in\partial D(0,1)}g(z\cdot w))\sigma(dw)=\int_{z\in\partial D(0,1)}g(w)\sigma(dw).
\]
Hence $\sigma$ is invariant relative to the left-multiplication by
an arbitrary group element $z\in\partial D_{0,1}$. 

12. In particular, since the uniform distribution $\overline{\sigma}_{m,0,1}$
is invariant relative to rotations, it is invariant relative to the
left-multiplication by an arbitrary group element $z\in\partial D_{0,1}$.
Hence theorem (1.19) in chapter 8 of \cite{BishopBridges85} says
that there exists a constant $c$ such that $\sigma=c\overline{\sigma}_{m,0,1}$.
Since $\sigma$ and $\overline{\sigma}_{m,y,r}$ are both distributions,
it follows that $1=\sigma(1)=c\overline{\sigma}_{m,0,1}(1)=c$. Therefore
$\sigma=\overline{\sigma}_{m,0,1}$, as alleged.
\end{proof}
\begin{cor}
\label{Cor. distribution on D(y,r) invariant relative to rotation is equal to uniform distribution-1}\textbf{\emph{
Each distribution on the $(m-1)$-sphere $\partial D_{y,r}$ that
is invariant relative to rotations is equal to the uniform distribution.
}}Let $y\in R^{m}$ and $r>0$ be arbitrary. Then the following conditions
hold.

1. Let $\sigma$ be an arbitrary distribution on $\partial D_{y,r}$
that is invariant relative to rotations about the center $y$. Then
$\sigma=\overline{\sigma}_{m,y,r}$. 

2. Suppose $y=0$. Define a distribution $\widetilde{\sigma}$ on
$\partial D_{0,r}$ by 
\[
\widetilde{\sigma}(f)\equiv\int_{z\in\partial D}\overline{\sigma}_{m,0,1}(dz)f(rz)
\]
for each $f\in C(\partial D_{0,r})$. Then $\widetilde{\sigma}=\overline{\sigma}_{m,0,r}$. 

3. A function $f$ on $\partial D_{0,r}$ is integrable relative to
$\overline{\sigma}_{m,0,r}$ iff $f(r\cdot)$ is integrable relative
to $\overline{\sigma}_{m,0,1}$, in which case 
\[
\int_{x\in\partial D(0,r)}\overline{\sigma}_{m,0,r}(dx)f(x)=\int_{z\in\partial D}\overline{\sigma}_{m,0,1}(dz)f(rz).
\]
\end{cor}

\begin{proof}
1. Let $\sigma$ be an arbitrary distribution on $\partial D_{y,r}$
that is invariant relative to rotations about the center $y$. Define
the distribution $\widetilde{\sigma}$ on $\partial D_{0,1}$ by $\widetilde{\sigma}(g)\equiv\sigma(g(\frac{\cdot-y}{r}))$
for each $g\in C(\partial D_{0,1})$. Then $\widetilde{\sigma}$ is
a distribution on $\partial D_{0,1}$ that is invariant relative to
rotations about $0$. Similarly, the distribution $\widetilde{\overline{\sigma}}_{m,y,r}$
is a distribution on $\partial D_{0,1}$ that is invariant relative
to rotations about $0$. By Theorem \ref{Thm. distribution on D(0,1) invariant relative to rotation is equal to uniform distribution},
we have $\widetilde{\sigma}=\widetilde{\overline{\sigma}}_{m,y,r}$.
Now let $h\in C(\partial D_{y,r})$ be arbitrary. Define $g\in C(\partial D_{0,1})$
by $g\equiv h(y+r\cdot)$. Then
\[
\sigma(h)=\sigma(g(\frac{\cdot-y}{r}))\equiv\widetilde{\sigma}(g)=\widetilde{\overline{\sigma}}_{m,y,r}(g)=\overline{\sigma}_{m,y,r}(g(\frac{\cdot-y}{r}))=\overline{\sigma}_{m,y,r}(h),
\]
where $h\in C(\partial D_{y,r})$ is arbitrary. Thus $\sigma\equiv\overline{\sigma}_{m,y,r}$
as distributions. Assertion 1 is proved.

2. Suppose $y=0$. Define a distribution $\widetilde{\sigma}$ on
$\partial D_{0,r}$ by 
\[
\widetilde{\sigma}(f)\equiv\int_{z\in\partial D}\overline{\sigma}_{m,0,1}(dz)f(rz)
\]
for each $f\in C(\partial D_{0,r})$. Then $\widetilde{\sigma}$ is
a distribution on the compact space $\partial D_{0,r}$. Let $\alpha$
be an arbitrary $m\times m$ rotation matrix. Let $f\in C(\partial D_{0,r})$
be arbitrary. Define $g\equiv f(r\cdot)\in C(\partial D_{0,1})$.
Then
\[
\widetilde{\sigma}(f\circ\alpha)=\int_{z\in\partial D}\overline{\sigma}_{m,0,1}(dz)g\circ\alpha(z)
\]
\[
\equiv\overline{\sigma}_{m,0,1}(g\circ\alpha)=\overline{\sigma}_{m,0,1}(g)=\overline{\sigma}_{m,0,1}(f(r\cdot))\equiv\widetilde{\sigma}(f).
\]
where the third inequality is by Definition \ref{Def. Uniform distribttionn on (m-1)-sphere}.
Hence $\widetilde{\sigma}=\overline{\sigma}_{m,0,r}$ by Assertion
1. This proves Assertion 2.

3. By Assertion 2 
\begin{equation}
\int_{x\in\partial D(0,r)}\overline{\sigma}_{m,0,r}(dx)f(x)\equiv\int_{z\in\partial D}\overline{\sigma}_{m,0,1}(dz)f(rz)\label{eq:temp-56}
\end{equation}
for each $f\in C(\partial D_{0,r})$. By $L_{1}$ continuity, equality
\ref{eq:temp-56} can be extended to each function $f$ on $\partial D_{0,r}$
that is integrable relative to $\overline{\sigma}_{m,0,r}$. Conversely,
equality \ref{eq:temp-56} can be extended to each function $f$ on
$\partial D_{0,r}$ such that the function $f(r\cdot)$ is integrable
relative to $\overline{\sigma}_{m,0,1}$. Assertion 3 is proved.
\end{proof}

\subsection{Harmonic function}
\begin{defn}
\textbf{\label{Def. Harmoncic function} Harmonic function.} Let $A$
be an arbitrary open subset of $R^{m}$. A function $u\equiv u(x_{1},\cdots,x_{m})$
that has a second derivative which is uniformly continuous on compact
subsets of $A$ is said to be \emph{harmonic} if it satisfies the
Laplace equation
\[
\triangle u\equiv\sum_{i=1}^{m}\frac{\partial^{2}u}{\partial x_{i}^{2}}=0
\]
on $A$.
\end{defn}

$\square$
\begin{defn}
\textbf{\label{Def. Mean value property }Mean value property}. Let
$A$ be an arbitrary open subset of $R^{m}$. Let $u$ be a function
with $domain(u)=A$ that is uniformly continuous on compact subsets
of \emph{$A$. }Then the function $u$ is said to have the mean value
property if one of the following two conditions hold.

(i) (Mean value property on balls). For each $y\in R^{m}$ and $r>0$
such that $\overline{D}_{y,r}\subset A$, we have

\begin{equation}
u(y)=\nu_{m,r}^{-1}\int\cdots\int_{x\in D(y,r)}u(x)dx.\label{eq:temp-3-2}
\end{equation}

(ii) (Mean value property on spheres). For each $y\in R^{m}$ and
$r>0$ such that $\overline{D}_{y,r}\subset A$, we have

\begin{equation}
u(y)=\sigma_{m,r}^{-1}\int_{z\in\partial D(y,r)}u(z)\sigma_{m,y,r}(dz).\label{eq:temp-5}
\end{equation}
The next proposition says that these two conditions are equivalent.
$\square$
\end{defn}

\begin{prop}
\label{Prop. mean value property on balls is equivalent to mean value property on spheres.}\textbf{\emph{
Mean value property on balls is equivalent to mean value property
on spheres.}} Let $A$ be an arbitrary open subset of $R^{m}$.Let
$u$ be a function with $domain(u)=A$ that is uniformly continuous
on compact subsets of \emph{$A$. Then Conditions (i) and (ii) in
Definition \ref{Def. Mean value property } are equivalent.}
\end{prop}

\begin{proof}
1. Suppose Condition (i) holds for the function $u$. Let $y\in R^{m}$
and $r>s>0$ be arbitrary such that $\overline{D}_{y,r}\subset A$.
Then 
\begin{equation}
\nu_{m,r}\cdot u(y)=\int\cdots\int_{x\in D(y,r)}u(x)dx,\label{eq:temp-13}
\end{equation}
with a similar equality where $r$ is replaced by $s$. Hence
\[
u(y)=\lim_{s\uparrow r}\nu_{m,s,r}^{-1}(\nu_{m,r}-\nu_{m,s})\cdot u(y)
\]
\[
=\lim_{s\uparrow r}\nu_{m,s,r}^{-1}(\int\cdots\int_{x\in D(y,r)}u(x)dx-\int\cdots\int_{x\in D(y,s)}u(x)dx)
\]
\[
=\lim_{s\uparrow r}\nu_{m,s,r}^{-1}\int\cdots\int_{x\in D(y,s,r)}u(x)dx
\]
\[
=\lim_{s\uparrow r}\nu_{m,s,r}^{-1}\int\cdots\int_{x\in D(0,s,r)}u(x+y)dx
\]
\[
=\lim_{s\uparrow r}\nu_{m,s,r}^{-1}\int\cdots\int_{x\in D(0,s,r)}u(\frac{rx}{\left\Vert x\right\Vert }+y)dx
\]
\[
=\lim_{s\uparrow r}\sigma_{m,r}^{-1}\int_{z\in D(0,r)}\sigma_{m,0,r}(dz)u(z+y)
\]
\[
\equiv\sigma_{m,r}^{-1}\int_{z\in D(0,r)}\sigma_{m,0,r}(dz)u(z+y)
\]
\[
=\sigma_{m,r}^{-1}\int_{x\in D(y,r)}\sigma_{m,y,r}(dx)u(x),
\]
where the second equality is due to equality \ref{eq:temp-13}, where
the fifth equality is because $\left\Vert x\right\Vert \rightarrow r$
for each $x\in D(0,s,r)$ as $s\uparrow r$ and because $u$ is uniformly
continuous on $x\in D(0,s,r),$ and where the sixth equality is by
applying equality \ref{eq:temp-12} of Proposition \ref{Prop:-Surface-area-integration on sphere centered at 0}
to the function $g\equiv u(\cdot+y)$.

Thus Condition (ii) holds. We conclude that Condition (i) implies
Condition (ii).

2. Conversely, suppose Condition (ii) in Definition \ref{Def. Mean value property }
holds. Let $y\in R^{m}$ and $r>s>0$ be arbitrary such that $\overline{D}_{y,r}\subset A$.
Let $\varepsilon>0$ be arbitrary. Take $K\geq1$ so large that $rs^{-1}K^{-1}(r-s)<\delta_{u}(\varepsilon)$
where $\delta_{u}$ is a modulus of continuity of the function $u$
on $\overline{D}_{y,r}$. Let $r_{k}\equiv s+kK^{-1}(r-s)$ for each
$k=0,\cdots,K$. Then
\[
\nu_{m,s,r}^{-1}\int\cdots\int_{x\in D(0,s,r)}u(y+x)dx=\nu_{m,s,r}^{-1}\sum_{k=1}^{K}\int\cdots\int_{x\in D(0,r(k-1),r(k))}u(y+x)dx.
\]
At the same time, for each $k=1,\cdots,K$ and $x\subset D_{0,r(k-1),r(k)}$,
we have
\[
\left\Vert (y+x)-(y+r_{k}\frac{x}{\left\Vert x\right\Vert })\right\Vert \leq\left\Vert x\right\Vert \cdot|1-\frac{r_{k}}{r_{k-1}}|
\]
\[
\leq r\frac{r_{k}-r_{k-1}}{r_{k-1}}\leq r\frac{K^{-1}(r-s)}{s}<\delta_{u}(\varepsilon).
\]
Hence 
\[
\nu_{m,s,r}^{-1}\int\cdots\int_{x\in D(0,s,r)}u(y+x)dx
\]
\[
=\nu_{m,s,r}^{-1}\sum_{k=1}^{K}\int\cdots\int_{x\in D(0,r(k-1),r(k))}u(y+x)dx
\]
\[
=\nu_{m,s,r}^{-1}\sum_{k=1}^{K}\int\cdots\int_{x\in D(0,r(k-1),r(k))}(u(y+r_{k}\frac{x}{\left\Vert x\right\Vert })\pm\varepsilon)dx
\]
\begin{equation}
=\nu_{m,s,r}^{-1}\sum_{k=1}^{K}\int\cdots\int_{x\in D(0,r(k-1),r(k))}u(y+r_{k}\frac{x}{\left\Vert x\right\Vert })dx\pm\varepsilon.\label{eq:temp-50}
\end{equation}
By applying equality \ref{eq:temp-12} of Proposition \ref{Prop:-Surface-area-integration on sphere centered at 0}
with $u(y+\cdot)$ in the place of the function $g$ and with $r_{k-1}$
and $r_{k}$ in the places of $s$ and $r$ respectively, we obtain
\[
\int\cdots\int_{x\in D(0,r(k-1),r(k)))}u(y+\frac{r_{k}x}{\left\Vert x\right\Vert })dx
\]
\begin{equation}
=\nu_{m,r(k-1),r(k)}\sigma_{m,r(k)}^{-1}\sigma_{m,0,r(k)}(u(y+\cdot)).\label{eq:temp-12-2}
\end{equation}
Therefore equality \ref{eq:temp-50} reduces to
\[
\nu_{m,s,r}^{-1}\int\cdots\int_{x\in D(0,s,r)}u(y+x)dx
\]
\[
=\nu_{m,s,r}^{-1}\sum_{k=1}^{K}\nu_{m,r(k-1),r(k)}\sigma_{m,r(k)}^{-1}\sigma_{m,0,r(k)}(u(y+\cdot))\pm\varepsilon
\]
\[
=\nu_{m,s,r}^{-1}\sum_{k=1}^{K}\nu_{m,r(k-1),r(k)}\sigma_{m,r(k)}^{-1}\sigma_{m,y,r(k)}(u)\pm\varepsilon
\]
\begin{equation}
=\nu_{m,s,r}^{-1}\sum_{k=1}^{K}\nu_{m,r(k-1),r(k)}u(y)\pm\varepsilon=u(y)\pm\varepsilon,\label{eq:temp-10}
\end{equation}
where the second to last equality is thanks to the assumed Condition
(ii) in Definition \ref{Def. Mean value property }. Since $\varepsilon>0$
is arbitrarily small, we infer that 
\[
\nu_{m,s,r}^{-1}\int\cdots\int_{x\in D(0,s,r)}u(y+x)dx=u(y).
\]
Letting $s\downarrow0$, we obtain
\[
u(y)=\nu_{m,0,r}^{-1}\int\cdots\int_{x\in D(0,,r)}u(y+x)dx=\nu_{m,r}^{-1}\int_{x\in D(y,r)}u(x)dx,
\]
which is Condition (i) in Definition \ref{Def. Mean value property }.
Thus Condition (i) follows from Condition (ii).
\end{proof}
\begin{prop}
\emph{\label{Prop. Harmonic functions have mean value property}}\textbf{\emph{
Harmonic functions}}\textbf{ }\textbf{\emph{have mean value property}}\textbf{.
}Let $A$ be an arbitrary open subset of $R^{m}$. Suppose a function
$u\equiv u(x)\equiv u(x_{1},\cdots,x_{m})$ is harmonic on $A$. Suppose
a ball $\overline{D}_{y,r}\subset A$ for some $y\in R^{m}$ and $r>0$
. Then 

\begin{equation}
u(y)=\nu_{m,r}^{-1}\int_{x\in D(y,r)}u(x)dx.\label{eq:temp-3}
\end{equation}
Moreover,
\begin{equation}
u(y)=\sigma_{m,r}^{-1}\sigma_{m,y,r}(u)=\int_{z\in\partial D(y,r)}u(z)\overline{\sigma}_{m,y,r}(dz).\label{eq:temp-4}
\end{equation}
\end{prop}

\begin{proof}
See theorems 1.5 and 1.6 of \cite{Helms69}.
\end{proof}
\begin{lem}
\emph{\label{Lem. C^k function smoothed by convolution with indicator of ball is C^(k+1)}
}\textbf{\emph{$C^{(k)}$-function smoothed by convolution with indicator
of a ball is $C^{(k+1)}$}}. Let $u$ be an arbitrary bounded continuous
function on the open set $A\subset R^{m}$. Let $r>0$ be arbitrary.
Define the function $u_{r}$ on the subset
\[
A_{r}\equiv\{y\in A:\overline{D}_{y,r}\subset A\}
\]
by
\begin{equation}
u_{r}(y)\equiv\nu_{m,r}^{-1}\int_{\overline{D}(y,r)}u(x)dx\label{eq:temp-5-1}
\end{equation}
for each $y\in A_{r}$. Then the following conditions hold.

1. $u_{r}$ has continuous first derivative on $A_{r}$. 

2. Suppose, in addition, the function $u$ has continuous $k$-th
derivative on $A.$ for some $k\geq1$. Then the function $u_{r}$
has continuous ($k+1)$-st derivative on $A_{r}$.

3. Suppose $u$ has the mean value property. Then $u$ has continuous
($k+1)$-st derivative on $A_{r}$ for each $k\geq1$.
\end{lem}

\begin{proof}
1. For the proof of Assertions 1 and 2, see theorem 1.14 on page 19
of \cite{Helms69}

2. To prove Assertion 3, suppose $u$ has the mean value property.
In other words, suppose $u=u_{r}$ for each $y\in A_{r}$. Then, Assertion
1 implies that $u$ has continuous first derivative on $A_{r}$. Assertion
2 then implies that $u=u_{r}$ has continuous second derivative on
$A_{r}$. Recursively applying Assertion 2, we see that $u=u_{r}$
has continuous ($k+1)$-st derivative on $A_{r}$ for each $k\geq1$.
\end{proof}
\begin{thm}
\emph{\label{Thm. Divergence Theorem}}\textbf{\emph{ Divergence Theorem.}}
Let $A$ be an arbitrary open subset of $R^{m}$. Let $\mathbf{F}\equiv(F_{1},\cdots,F_{m}):A\rightarrow R^{m}$
be a vector field with continuous derivatives on each closed ball
contained in $A$. Then for each $y\in R^{m}$ and $r>0$ with $\overline{D}_{y,r}\subset A$,
we have defined\textbf{\emph{
\[
\int_{x\in D(y,r)}(\nabla\cdot\mathbf{F)}(x)dx=\int_{x\in\partial D(y,r)}(\mathbf{F}\cdot\mathbf{n})(x)\sigma_{m,y,r}(dx),
\]
}}\emph{where, for each $x\equiv(x_{1},\cdots,x_{m})\in\partial D_{y,r}$,
we have defined (i) $\mathbf{n}(x)\equiv(n_{1},\cdots,n_{m})(x)\equiv r^{-1}(x-y)$
is the outward-pointing unit normal vector at $x$, (ii) ($\nabla\cdot\mathbf{F)}(x)\equiv(\frac{\partial F_{1}}{\partial x_{1}}+\cdots+\frac{\partial F_{m}}{\partial x_{m}})(x)$,
and (iii) $(\mathbf{F}\cdot\mathbf{n})(x)\equiv(F_{1}n_{1}+\cdots+F_{m}n_{m})(x)$.}
\end{thm}

\begin{proof}
The reader should refer to \cite{Rudin76} pp 253-275 for a treatment
of differential forms and Stokes theorem and the divergence theorem
in $R^{n}$, where $n\geq1$ is arbitrary. Then the reader should
refer to theorem 10.51 on page 288 of \cite{Rudin76} which translates
the divergence theorem in differential forms to the present divergence
theorem in rectangular coordinates. The proof of theorem 10.51 in
\cite{Rudin76} assumes $m=3$. The reader should verify that the
assumption that $m=3$ is however not essential to said proof.
\end{proof}
\begin{prop}
\emph{\label{Prop. mean value property implies harmonic}}\textbf{\emph{
Functions with the mean value property are harmonic}}. Let $A$ be
an arbitrary open subset of $R^{m}$. Suppose a function $u\equiv u(x)\equiv u(x_{1},\cdots,x_{m})$
is uniformly continuous on each closed ball contained in $A$ and
\begin{equation}
u(y)=\int_{z\in\partial D(y,r)}u(z)\overline{\sigma}_{m,y,r}(dz)\label{eq:temp-4-1}
\end{equation}
for each $y\in R^{m}$ and $r>0$ with $\overline{D}_{y,r}\subset A$.
Then $u$ is harmonic on $A$.
\end{prop}

\begin{proof}
1. By equality \ref{eq:temp-4-1} in the hypothesis, the function
$u$ satisfies Condition (ii) in Definition \ref{Def. Mean value property }.
Therefore, by Proposition \ref{Prop. mean value property on balls is equivalent to mean value property on spheres.},
the function $u$ satisfies Condition (i) in Definition \ref{Def. Mean value property }.
Hence, according to Assertion 3 of Lemma \ref{Lem. C^k function smoothed by convolution with indicator of ball is C^(k+1)},
the function $u$ has a second derivative which is uniformly continuous
on each compact subset of $A$. In particular, $\triangle u$ is uniformly
continuous on each compact subset of $A$. 

2. Suppose, for the sake of a contradiction, that $\triangle u(y)<0$
for some $y\equiv(y_{1},\cdots,y_{m})\in A$. Then there exists $r>0$
with $\overline{D}_{y,r}\subset A$ such that $\triangle u<0$ on
$D_{y,r}$. Recall the function $\mathbf{F}:\overline{D}_{y,r}\rightarrow R^{m}$
defined by
\[
\mathbf{F}(x)\equiv(F_{1}(x),\cdots,F_{m}(x))\equiv\nabla u(x)\equiv(\frac{\partial u}{\partial x_{1}}(x),\cdots,\frac{\partial u}{\partial x_{m}}(x))
\]
for each $x\in\overline{D}_{y,r}$. Recall the functions $\mathbf{n}:\{x\in\overline{D}_{y,r}:\left\Vert x-y\right\Vert >0\}\rightarrow R^{m}$
by
\[
\mathbf{n}(x)\equiv\left\Vert x-y\right\Vert ^{-1}(x-y)
\]
for each $x\in\{x\in\overline{D}_{y,r}:\left\Vert x-y\right\Vert >0\}$.

Consider each $s\in(0,r]$. Then Condition (ii) in Definition \ref{Def. Mean value property }
implies that
\[
u(y)=\int_{x\in D(0,s)}\overline{\sigma}_{m,0,s}(dx)u(y+x)
\]
\[
=\int_{z\in\partial D(0,1)}\overline{\sigma}_{m,0,1}(dz)u(y+sz)
\]
where the second equality is by Assertion 1 of Proposition \ref{Prop:-Surface-area-integration on sphere centered at 0}.
With $y$ fixed, differentiation relative to $s$ under the integral
sign yields
\[
0=\sigma_{m,1}^{-1}\int_{z\in\partial D(0,1)}\sum_{i=1}^{m}\frac{\partial u}{\partial x_{i}}(y+sz)z_{i}\sigma_{m,0,1}(dz)
\]
\[
\equiv\sigma_{m,1}^{-1}\int_{z\in\partial D(0,1)}\mathbf{F}(y+sz)\cdot\mathbf{n}(y+sz)\sigma_{m,0,1}(dz)
\]
\[
\equiv\sigma_{m,1}^{-1}s^{-m+1}\int_{z\in\partial D(0,s)}\mathbf{F}(y+z)\cdot\mathbf{n}(y+z)\sigma_{m,0,s}(dz)
\]
\[
=\sigma_{m,1}^{-1}s^{-m+1}\int_{z\in\partial D(y,s)}\mathbf{F}(z)\cdot\mathbf{n}(z)\sigma_{m,y,s}(dz)
\]
\[
=\sigma_{m,1}^{-1}s^{-m+1}\int_{x\in D(y,s)}(\nabla\cdot\mathbf{F)}(x)dx
\]
\begin{equation}
=\sigma_{m,1}^{-1}s^{-m+1}\int_{x\in D(y,s)}\triangle u(x)dx<0,\label{eq:temp-6}
\end{equation}
where the third equality by applying Assertion 1 of Proposition \ref{Prop:-Surface-area-integration on sphere centered at 0}
to the function $g\equiv\mathbf{F}(y+s\cdot)\cdot\mathbf{n}(y+s\cdot)$,
where the fourth equality is by Definition \ref{Def. Uniform distribttionn on (m-1)-sphere},
and where the fifth equality is by the Divergence Theorem, Theorem
\ref{Thm. Divergence Theorem}. Inequality \ref{eq:temp-6} is a contradiction.
We conclude that $\triangle u(y)\geq0$ for each $y\in A$. By symmetry
$\triangle u(y)\leq0$ for each $y\in A$. Thus $\triangle u=0$ on
$A$. In other words, the function $u$ is harmonic on $A$.
\end{proof}
\begin{thm}
\textbf{\emph{\label{Thm. Maximum moduus thorems} Maximum modulus
theorem for harmonic functions.}} Let $A$ be an arbitrary open subset
of $R^{m}$. Let $u$ be an arbitrary harmonic function on $A$. Suppose
$n\geq1$, $y_{1},\cdots,y_{n}\in A$ and $\rho_{1},\cdots,\rho_{n}>0$
are such that (i) $K\equiv\bigcup_{i=1}^{n}\overline{D}_{y(i),\rho(i)}$
is a compact subset of $A$, and (ii) $D_{y(i),\rho(i)}\cap D_{y(i+1),\rho(i+1)}$
is nonempty for each $i=1,\cdots,n-1$. Then 
\[
\sup_{x\in K}|u(x)|=\sup_{x\in\partial K}|u(x)|.
\]
\end{thm}

\begin{proof}
This is a special case of corollary 1.13 of \cite{Helms69}.
\end{proof}
\begin{thm}
\emph{\label{Thm.  Poisson kernel and Poisson integration on the m-sphere}}\textbf{\emph{
Poisson kernel and Poisson integration on the $(m-1)$-sphere.}}\emph{
}Let $y\in R^{m}$ and $r>0$ be arbitrary. Define the function $\mathbf{k}_{y,r}:D_{y,r}\times\partial D_{y,r}\rightarrow(0,\infty)$
by 
\[
\mathbf{k}_{y,r}(x,z)\equiv\frac{1}{r}\frac{r^{2}-\left\Vert y-x\right\Vert ^{2}}{\left\Vert z-x\right\Vert ^{m}}
\]
for each $(x,z)\in D_{y,r}\times\partial D_{y,r}.$ The function $\mathbf{k}_{y,r}$
is called the \emph{Poisson kernel}\index{Poisson kernel} for the
ball $D_{y,r}$. Then the following conditions hold.

\emph{1.} Let $x\in D_{y,r}$ be arbitrary. Then $\mathbf{k}_{y,r}(x,\cdot)>0$
on $\partial D_{y,r},and$ 
\begin{equation}
1=\int_{z\in\partial D(y,r)}\mathbf{k}_{y,r}(x,z)\overline{\sigma}_{m,y,r}(dz).\label{eq:temp}
\end{equation}

2. Let $z\in\partial D_{y,r}$ be arbitrary. Then the function $\mathbf{k}_{y,r}(\cdot,z)$
is harmonic on the open ball $D_{y,r}.$ 

3. Let $x\in\overline{D}_{y,r}$be arbitrary. Then the function $\mathbf{k}_{y,r}(x,\cdot)$
is integrable relative to the uniform distribution $\overline{\sigma}_{m,y,r}$.

4. Let $\overline{g}\in C(\partial D_{y,r})$ be arbitrary. Then the
function
\begin{equation}
g(x)\equiv\int_{z\in\partial D(y,r)}\mathbf{k}_{y,r}(x,z)\overline{g}(z)\overline{\sigma}_{m,y,r}(dz)\label{eq:temp-628-1}
\end{equation}
of $x\in D_{y,r}$ is harmonic on $D_{y,r}$.

5. Let $\overline{g}\in C(\partial D_{y,r})$ be arbitrary and let
$g\in C(D_{y,r})$ be the function defined by equality \ref{eq:temp-628-1}.
Then\emph{ $\lim_{x\rightarrow z}g(x)=\overline{g}(z)$} uniformly
for each $z\in\partial D_{y,r}$. 

6. Let $x\in D_{y,r}$ be arbitrary. Define the function $\mu_{x}$
on $C(\partial D_{y,r})$ by
\begin{equation}
\mu_{x}g\equiv\overline{\sigma}_{m,y,r}\mathbf{k}_{y,r}(x,\cdot)g\label{eq:temp-45}
\end{equation}
for each $g\in C(\partial D_{y,r})$. Then the triple ($\partial D_{y,r},C(\partial D_{y,r}),\mu_{x})$
is an integration space with a complete extension $(\partial D_{y,r},\overline{L}_{x},\mu_{x})$,
where $\overline{L}_{x}$ is the space of integrable functions. 

7. Let $f$ be an arbitrary integrable function on $\partial D_{y,r}$
relative to the uniform distribution $\overline{\sigma}_{m,y,r}$.
Then the function $u$ defined on $D_{y,r}$ by 
\[
\overline{u}(x)\equiv\mu_{x}f\equiv\overline{\sigma}_{m,y,r}\mathbf{k}_{y,r}(x,\cdot)f
\]
for each $x\in D_{y,r}$ is a well defined harmonic function on $D_{y,r}$. 

8. Let $s\in(0,1)$ be arbitrary. Let $(x,z)\in\partial D_{0,s}\times\partial D_{0,1}$
be arbitrary. Then 
\begin{equation}
\mathbf{k}_{0,1}(x,z)\geq\frac{1-s^{2}}{(1+s)^{m}}\label{eq:temp-27-1-2-1}
\end{equation}
\end{thm}

\begin{proof}
1 Assertion 1 follows from of theorem 1.8 in \cite{Helms69}, where
the function $h$ is replaced by $1$.

2. Assertion 2 is a special case of theorem 2.3 in \cite{Helms69}
where the signed measure $\mu$ is replaced by the probability measure
$\delta_{z}$ which assigns probability 1 to the point $z$.. 

3. Assertion 3 is a special case of corollary 2.4 in \cite{Helms69}
where the signed measure $\mu$ is replaced by the surface-area integration
$\sigma_{m,y,r}$ on $\partial D_{m,y,r}$. 

4. Assertion 4 is, in essence, theorem 2.8 in \cite{Helms69}. 

5. Assertion 5 can be verified from the proofs of the lemmas 2.5-7
in \cite{Helms69}.

6. To prove Assertion 6, let $x\in D_{y,r}$ be arbitrary. Let $g\in C(\partial D_{y,r})$
be arbitrary. Then $g$ is bounded and continuous. By Assertion 3,
$\mathbf{k}_{y,r}(x,\cdot)$ is integrable relative to $\overline{\sigma}_{m,y,r}$.
Hence the product $\mathbf{k}_{y,r}(x,\cdot)g$ is integrable relative
to $\overline{\sigma}_{m,y,r}$, and $\mu_{x}g$ is well defined.
Moreover, since $\mathbf{k}_{y,r}(x,\cdot)>0$, it follows that $\mu_{x}$
is a linear function that satisfies conditions (i) and (ii) in definition
4.2.1 of \cite{Chan21} to be an integration on the compact space
$\partial D_{y,r}$. Hence, according to proposition 4.3.3 of \cite{Chan21},
the triple $(\partial D_{y,r},C(\partial D_{y,r}),\mu_{x})$ is an
integration space. Therefore, by proposition 4.4.2 of \cite{Chan21},
said triple can be extended to a complete integration space $(\partial D_{y,r},\overline{L}_{x},\mu_{x})$.
Assertion 6 is proved. 

7. It remains to prove Assertion 7. By hypothesis $f$ is an integrable
function on $\partial D_{y,r}$ relative to the uniform distribution
$\overline{\sigma}_{m,y,r}$. Therefore there exists a sequence $(g_{k})_{k=1.2.\cdots}$
in $C(\partial D_{y,r})$ such that
\begin{equation}
\sum_{k=1}^{\infty}\overline{\sigma}_{m,y,r}\left|g_{k}\right|<\infty,\label{eq:temp-52}
\end{equation}
\begin{equation}
domain(f)=\{z\in\partial D_{y,r}:\sum_{k=1}^{\infty}\left|g_{k}(z)\right|<\infty\},\label{eq:temp-53}
\end{equation}
and 
\begin{equation}
f(z)=\sum_{k=1}^{\infty}g_{k}(z)\label{eq:temp-59}
\end{equation}
for each $z\in domain(f)$. In short, the sequence $(g_{k})_{k=1.2.\cdots}$
is a representation of the integrable function $f$ relative to $\overline{\sigma}_{m,y,r}$. 

Let $s\in(0,r)$ be arbitrary. Consider each $\overline{x}\in\overline{D}_{y,s}\subset D_{y,r}$
and $k\geq1$. Then $\mathbf{k}_{y,r}(\overline{x},\cdot)\in C(\partial D_{y,r})$
with $b_{s}\geq\mathbf{k}_{y,r}(\overline{x},\cdot)\geq0$ for some
$b_{s}\geq0$. It follows that $\mathbf{k}_{y,r}(\overline{x},\cdot)g_{k}\in C(\partial D_{y,r})$.
According to equality \ref{eq:temp-45}, for each $x\in D_{y,r}$
the integral
\begin{equation}
\mu_{x}g_{k}\equiv\overline{\sigma}_{m,y,r}\mathbf{k}_{y,r}(x,\cdot)g_{k}\label{eq:temp-45-3}
\end{equation}
is well defined and, according to Assertion 4, is a harmonic function
of $x\in D_{y,r}$. Similarly,
\[
\sum_{k=1}^{\infty}\mu_{\overline{x}}\left|g_{k}\right|\equiv\sum_{k=1}^{\infty}\overline{\sigma}_{m,y,r}\mathbf{k}_{y,r}(\overline{x},\cdot)\left|g_{k}\right|
\]
\begin{equation}
\leq b_{s}\sum_{k=1}^{\infty}\overline{\sigma}_{m,y,r}\left|g_{k}\right|<\infty.\label{eq:temp-60}
\end{equation}
where the last inequality is by inequality \ref{eq:temp-52}. Combining
inequality \ref{eq:temp-60} with equalities \ref{eq:temp-53} and
\ref{eq:temp-59}, we see that the sequence $(g_{k})_{k=1.2.\cdots}$
in $C(\partial D_{y,r})$ is a representation of the function $f$
relative to the integration $\mu_{\overline{x}}$ on $\partial D_{y,r}$.
Hence the function $f$ is integrable relative to the integration
$\mu_{\overline{x}}$. Moreover, for each $\overline{x}\in D_{y,s}$
and for each $n\geq1$ we have 
\[
|\mu_{\overline{x}}f-\sum_{k=1}^{n}\mu_{\overline{x}}g_{k}|=|\sum_{k=1}^{\infty}\mu_{\overline{x}}g_{k}-\sum_{k=1}^{n}\mu_{\overline{x}}g_{k}|
\]
\[
=|\sum_{k=n+1}^{\infty}\mu_{\overline{x}}g_{k}|\leq|\sum_{k=n+1}^{\infty}\overline{\sigma}_{m,y,r}\mathbf{k}_{y,r}(\overline{x},\cdot)g_{k}|
\]
\[
\leq b_{s}\sum_{k=n+1}^{\infty}\overline{\sigma}_{m,y,r}|g_{k}|.
\]
Note that the last bound depends on $s$ but is otherwise independent
of $\overline{x}\in D_{y,s}$. In other words, $\sum_{k=1}^{n}\mu_{\overline{x}}g_{k}\rightarrow\mu_{\overline{x}}f$
uniformly in $\overline{x}\in D_{y,s}$, where $s\in(0,r)$ is arbitrary.
Thus the harmonic function $v(x)\equiv\sum_{k=1}^{n}\mu_{x}g_{k}$
of $x\in D_{y,r}$ converges to the function $\overline{u}(x)\equiv\mu_{x}f$
uniformly on the compact subset $D_{y,s}$ for each $s\in(0,r)$,
as $n\rightarrow\infty$. At the same time, the harmonic function
$v$ has the mean value property on compact subsets of $D_{y,r}$.
Therefore, in view of the aforementioned uniform convergence, the
function $\overline{u}$ inherits the mean value property on compact
subsets of $D_{y,r}.$ It follows from Proposition \ref{Prop. mean value property implies harmonic}
that the function $\overline{u}$ is a harmonic function on $D_{y,r}$.
Assertion 7 is proved.

8. Let $s\in(0,1)$ be arbitrary. We compute 
\[
\bigwedge_{(x,z)\in\partial D_{0,s}\times\partial D_{0,1}}\mathbf{k}_{0,1}(x,z)=\bigwedge_{v\in\partial D_{0,s}}\mathbf{k}_{0,1}(v,1)
\]
\[
=\bigwedge_{v\in\partial D_{0,s}}\frac{1-\left\Vert v\right\Vert ^{2}}{\left\Vert 1-v\right\Vert ^{m}}=\bigwedge_{v\in\partial D_{0,s}}\frac{1-s^{2}}{\left\Vert 1-v\right\Vert ^{m}}
\]
\begin{equation}
=\frac{1-s^{2}}{(1+s)^{m}},\label{eq:temp-27-1-2-1-2}
\end{equation}
where the first equality is because $\mathbf{k}_{0,1}$ is invariant
relative to a rotation $\alpha$ about the center $0$ that maps $z\in\partial D_{0,1}$
to $1\in\partial D_{0,1}$. Assertion 8 and the theorem are proved.
\end{proof}

\subsection{Brownian motion and exit time}
\begin{defn}
\label{Def. Specification of two Brownian motions}\textbf{ Specification
of two Brownian motions.} In the remaining of this article, let $m\geq1$
be arbitrary. Let $B,\widetilde{B}:[0,\infty)\times(\Omega,L,E)\rightarrow R^{m}$
be two independent Brownian motions, as defined in \cite{Chan21}.
Thus $B_{0}=\widetilde{B}_{0}=0\in R^{m}$. Let
\[
\mathscr{\mathcal{L}}\equiv\{L^{(t)}:t\geq0\}
\]
be the right continuous extension of the natural filtration of $B$.
For each $i=1,\cdots,m$ and $t\in[0,\infty)$ let $B_{i,t}$ denote
the r.r.v. that is the $i$-th component of $B_{t}$. Thus $B_{i,\cdot}:[0,\infty)\times(\Omega,L,E)\rightarrow R^{1}$
is a Brownian motion in $R^{1}$ for each $i=1,\cdots,m$. Here the
dot in the subscript of $B_{i,\cdot}$ serves as a place holder for
the time parameter $t$. 

Let $x\equiv(x_{1},\cdots,x_{m})\in R^{m}$ be arbitrary. Define the
Brownian motion $B^{x}$ \emph{with initial state} $x$ by $B^{x}\equiv x+B:[0,\infty)\times(\Omega,L,E)\rightarrow R^{m}$.
Thus $B_{t}^{x}=x+B_{t}.$ for each $t\geq0$. The process $B^{x}$
is adapted to the filtration $\mathcal{L}.$ Similarly define the
Brownian motion $\widetilde{B}^{x}\equiv x+\widetilde{B}$ with initial
state $x$.

Our main focus is on the process $B$; the process $\widetilde{B}$
merely facilitates some presentation.

Refer to \cite{Chan21} for the definition and basic properties of
a stopping time relative to the filtration $\mathscr{\mathcal{L}}.$
$\square$
\end{defn}

\begin{thm}
\emph{\label{Thm: Strong Markov property of Brownian motion}}\textbf{\emph{
Strong Markov property of Brownian motion.}} Let $x\in R^{m}$ be
arbitrary. Then $B^{x}:[0,\infty)\times(\Omega,L,E)\rightarrow R^{m}$
is a strong Markov process relative to the right continuous filtration
$\mathscr{\mathcal{L}}.$ Specifically, let $\tau:(\Omega,L,E)\rightarrow[0,\infty)$
be an arbitrary stopping time relative to the filtration $\mathscr{\mathcal{L}}.$
Define the probability subspace
\[
L^{(\tau)}\equiv\{Y\in L:Y1_{\tau\leq t}\in L^{(t)}\;\mathrm{for\;each\;regular\;point}\;t\;\mathrm{of}\;\tau\}
\]
of integrable observables up to and including the stopping time $\tau$.
Let $n\geq1$, the sequence $0\leq r_{0}\leq r_{1}\leq\cdots\leq r_{n}$,
and the function $g\in C_{ub}(R^{n+1})$ be arbitrary. Then
\[
E(g(B_{\tau+r(0)}^{x},\cdots,B_{\tau+r(n)}^{x})|L^{(\tau)})=Eg(\widetilde{B}_{r(0)}^{B^{x}(\tau)},\cdots,\widetilde{B}_{r(n)}^{B^{x}(\tau)}).
\]
In words, given all observables of the process $B^{x}$ up to and
including the stopping time $\tau$, the future development is as
if starting an independent Brownian motion $\widetilde{B}$ anew at
the current state $B_{\tau}^{x}$. 
\end{thm}

\begin{proof}
See \cite{Chan21}.
\end{proof}
\begin{lem}
\label{Lem. Reflectioin principle} \textbf{\emph{Reflection principle.}}
Suppose $m=1$. In other words, suppose $B$ is a real valued Brownian
motion. Let $t,\lambda>0$ be arbitrary. Then 
\[
P(\bigvee_{s\in[0,t]}B_{s}\geq\lambda)=2(1-\Phi_{0,1}(\frac{\lambda}{\sqrt{t}})),
\]
where $\Phi_{0,1}$ is the standard normal CDF on $R$.
\end{lem}

\begin{proof}
The specified Brownian motion $B$ has continuous paths. Therefore
theorem 2.1 in \cite{Doob53}, of which the present lemma is a restatement,
and the proof of said theorem 2.1 in \cite{Doob53} are applicable. 
\end{proof}
\begin{defn}
\label{Def. Modulus of continuity at 0 of standard normal CDF} \textbf{Modulus
of continuity at $0$ of the standard normal CDF.} It will be convenient
to fix a sequence of integers which represents the modulus of continuity
at $0$ of the standard normal CDF $\Phi_{0,1}$. Specifically, let
$\widetilde{r}>0$ be arbitrary. Then the function $2\Phi_{0,1}(\frac{\widetilde{r}}{\sqrt{t}})-1$
of $t\in(0,\infty)$ is strictly decreasing, and converges to $0$
as $t\uparrow\infty$, i.e. as $\frac{\widetilde{r}}{\sqrt{t}}\downarrow0$.
Hence, given an arbitrary $\widetilde{r}>0$, we can fix an increasing
sequence $(N_{\widetilde{r},k})_{k=0,1,\cdots}$ of positive integers
such that
\begin{equation}
2\Phi_{0,1}(\frac{\widetilde{r}}{\sqrt{N(\widetilde{r},k)}})-1<2^{-k}.\label{eq:temp-14}
\end{equation}
for each $k\geq0$.
\end{defn}

$\APLbox$
\begin{defn}
\label{Def. Finiteness of stopping times}\textbf{ Convention regarding
stopping times.} In the remainder of this article, we make the convention
that all stopping times in the discussion are r.r.v.'s with values
in $[0,\infty)$. Before each stopping time is used, there will be
a stated or unstated proof that it is a r.r.v. with values in $[0,\infty)$.
In other words, a stopping time $\tau$ is admissible only if $P(\tau\leq t)\uparrow1$
as $t\rightarrow\infty$. This is in contrast to some usage in the
literature where the point $\infty$ at infinity for the extended
time line $[0,\infty]$ is an admissible value for a stopping time.
\end{defn}

$\square$
\begin{defn}
\textbf{\label{Def First exit time from open ball by Brownian motion}
First exit time from open ball by Brownian motion. }Let $r>0$ and
$x\in D_{0,r}$ be arbitrary. Suppose $\tau$ is a stopping time relative
to $\mathcal{L}$, with values in $(0,\infty)$. Define the function
$B_{\mathcal{\tau}}^{x}:\Omega\rightarrow R^{m}$ by $domain(B_{\mathcal{\tau}}^{x})\equiv domain(\tau)$
and by
\[
B_{\mathcal{\tau}}^{x}(\omega)\equiv B^{x}(\tau(\omega),\omega)
\]
for each $\omega\in domain(\tau)$. Suppose, in addition, the following
three conditions hold.

(i) For each $\omega\in domain(\tau)$, we have $B_{\mathcal{\tau}}^{x}(\omega)\in\partial D_{0,r}$
.

(ii) The function $B_{\mathcal{\tau}}^{x}:(\Omega,L,E)\rightarrow\partial D_{0,r}$
is a r.v.

(iii) For each $\omega\in domain(\tau)$ and $t\in[0,\tau(\omega))$,
we have $B_{t}^{x}(\omega)\in D_{0,r}.$

Then we define $\tau_{x,r}\equiv\tau$ and call $\tau_{x,r}$ the
\emph{first exit time} for the process $B^{x}$ to exit the open ball
$D_{0,r}$, and write $\tau_{D(0,r)}(B^{x})\equiv\tau_{x,r}$, and
call the r.v. $B_{\mathcal{\tau}(x,r)}^{x}$ the corresponding \emph{random
point of exit}. If emphasis of the underlying Brownian motion $B$
is needed, then we write $\tau_{x,r;B}$ for $\tau_{x,r}$. $\square$
\end{defn}

\begin{lem}
\textbf{\emph{\label{Lem.  Existence of certain first exit tims }
Existence of certain first exit times.}} Let $\widetilde{r}>\overline{r}>0$
and $x\in D_{0,\overline{r}}$ be arbitrary. Then the following conditions
hold.

1. There exists a countable subset $H$ of $R$ such that, for each
$a\in(\overline{r},\widetilde{r})H_{c},$ the first exit time $\tau_{x,a}\equiv\tau_{D(0,a)}(B^{x})$
exists. Here $H_{c}$ denotes the metric complement of $H$ in $R$.
Thus $\tau_{x,\cdot}:(\overline{r},\widetilde{r})H_{c}\times\Omega\rightarrow(0,\infty)$
is a stochastic process with nondecreasing paths.

A subsequent theorem will prove that the process $\tau_{x,\cdot}$
is a.u. continuous on the parameter set $(\overline{r},\widetilde{r})H_{c}$,
which is dense in the interval $(\overline{r},\widetilde{r})$, Therefore
the process $\tau_{x,\cdot}$ is extendable to an a.u. continuous
process with parameter set $(\overline{r},\widetilde{r})$.

2. The family $\{\tau_{x,a}:x\in D_{0,\overline{r}};a\in(\overline{r},\widetilde{r})H_{c}\}$
is tight, with the sequence $(N_{\widetilde{r},k})_{k=0,1,\cdots}$
of positive integers playing the role of modulus of tightness. Specifically,
let $k\geq0$ and $t>N_{\widetilde{r},k}$, $x\in D_{0,\overline{r}},$
$a\in(\overline{r},\widetilde{r})H_{c}$ be arbitrary. Then
\begin{equation}
P(\tau_{x,a}>t)\leq2^{-k+1}.\label{eq:temp-58}
\end{equation}
\end{lem}

\begin{proof}
1. Let $\widetilde{r}>\overline{r}>0$ and $x\equiv(x_{1},\cdots,x_{m})\in D_{0,\overline{r}}$
be arbitrary. Thus $\left\Vert x\right\Vert \leq\overline{r}$. Define
the function $f:R^{m}\rightarrow R$ by 
\begin{equation}
f(z)\equiv\widetilde{r}\wedge\left\Vert z\right\Vert \label{eq:temp-2}
\end{equation}
for each $z\in R^{m}$. Then $f\in C_{ub}(R^{m})$ and $f(x)\leq\overline{r}.$
At the same time, a Brownian motion is a Feller process in the sense
of \cite{Chan21}. Hence, by theorem 11.10.7 of \cite{Chan21}, there
exists a countable subset $H$ of $R$ such that, for each 
\[
a\in(\overline{r},\widetilde{r})H_{c}
\]
and for each $n\geq1$, the first exit time $\tau\equiv\tau_{x,a,n}\equiv\mathcal{\tau}_{f,a,n}(B^{x})$
for the process $B^{x}$ to exit the open set 
\[
(f<a)\equiv\{z\in R^{m}:f(z)<a\}
\]
on or before time $n$, is well defined relative to the filtration
$\mathcal{L}$, in the sense of definition 10.11.1 of \cite{Chan21}.
Moreover, by definition 10.11.1 of \cite{Chan21}, the first exit
time $\tau_{x,a,n}$ has values in $(0,n]$. Consider each $a\in(\overline{r},\widetilde{r})H_{c}$.
Then, by said definition, the function
\[
\tau\equiv\tau_{x,a,n}\equiv\mathcal{\tau}_{f,a,n}(B^{x})
\]
is a stopping time relative to $\mathcal{L}$, with values in $(0,n]$,
and the function
\begin{equation}
B_{\mathcal{\tau}}^{x}\equiv B_{\tau(x,a,n)}^{x}:\Omega\rightarrow R^{m}\label{eq:temp-54}
\end{equation}
is a well-defined r.v. Furthermore, for each $\omega\in domain(\tau)$,
we have

(i) $f(B^{x}(\cdot,\omega))<a$ on the interval $[0,\tau(\omega))$,
and 

(ii)\emph{ $f(B_{\tau}^{x}(\omega))\geq a$ }if\emph{ $\tau(\omega)<n$.}

2. By lemma 10.11.2 of \cite{Chan21}, we have, for each $n'\geq n\geq1$
and $t\in(0,n)$,
\begin{equation}
\tau_{x,a,n}\leq\tau_{x,a,n'},\label{eq:temp-16}
\end{equation}
\begin{equation}
(\tau_{x,a,n}<n)\subset(\tau_{x,a,n'}=\tau_{x,a,n}),\label{eq:temp-17}
\end{equation}
and
\begin{equation}
(\tau_{x,a,n'}\leq t)=(\tau_{x,a,n}\leq t).\label{eq:temp-18}
\end{equation}

3. Note that $B^{x}\equiv(x_{1}+B_{1,\cdot},\cdots,x_{m}+B_{m,\cdot})$.
Here 
\[
B_{i,\cdot}:[0.\infty)\times\Omega\rightarrow R^{1}
\]
is the $i$-th component process of $B:[0.\infty)\times\Omega\rightarrow R^{m}$,
with the dot in the subscripts serving as a place holder for the time
parameter $t$. Let $n\geq1$ be arbitrary. Consider each $t\in(0,n)$.
Then 
\[
P(\tau_{x,a,n}>t)
\]
\[
\leq P\{\omega\in domain(\tau):f(B^{x}(\cdot,\omega))<a\;\mathrm{on}\:[0,t]\subset[0,\tau(\omega))\}
\]
\[
\leq P\{\omega:\bigvee_{s\in[0,t]}f(B^{x}(s,\omega))\leq a\;\mathrm{on}\:[0,t]\}
\]
\[
\equiv P\{\omega:\bigvee_{s\in[0,t]}\widetilde{r}\wedge\left\Vert B^{x}(s,\omega)\right\Vert \leq a\}
\]
\[
=P\{\omega:\bigvee_{s\in[0,t]}\left\Vert B^{x}(s,\omega)\right\Vert \leq a\}
\]
\[
\leq P\{\omega:\bigvee_{s\in[0,t]}\left\Vert B(s,\omega)\right\Vert \leq\left\Vert x\right\Vert +a\}
\]
\[
\leq P\{\omega:\bigvee_{s\in[0,t]}B_{1,\cdot}(s,\omega)\leq\left\Vert x\right\Vert +a\}
\]
\[
=2\Phi_{0,1}(\frac{\left\Vert x\right\Vert +a}{\sqrt{t}})-1
\]
\begin{equation}
\leq2\Phi_{0,1}(\frac{\overline{r}+\widetilde{r}}{\sqrt{t}})-1,\label{eq:temp-11}
\end{equation}
where the first inequality is by Condition (i), and where the last
equality is by the reflection principle, Lemma \ref{Lem. Reflectioin principle}
applied to the one-dimensional Brownian motion $B_{1,\cdot}$.

4. Next, define a function $\tau_{x,a}:\Omega\rightarrow(0,\infty)$
by 
\[
domain(\tau_{x,a})\equiv\{\omega\in\bigcap_{n=1}^{\infty}domain(\tau_{x,a,n}):\lim_{n\rightarrow\infty}\tau_{x,a,n}(\omega)\quad\mathrm{exists}\}
\]
and by
\[
\tau_{x,a}(\omega)\equiv\lim_{n\rightarrow\infty}\tau_{x,a,n}(\omega)
\]
for each $\omega\in domain(\tau_{x,a})$. 

5. We will verify that $\tau_{x,a}$ is a well defined r.r.v. and
is a stopping time relative to the filtration $\mathcal{L}$. Recall
from Definition \ref{Def. Modulus of continuity at 0 of standard normal CDF}
the increasing sequence $(N_{\widetilde{r},k})_{k=0,1,\cdots}$ of
positive integers, with
\begin{equation}
2\Phi_{0,1}(\frac{\widetilde{r}}{\sqrt{N(\widetilde{r},k)}})-1<2^{-k}.\label{eq:temp-14-1}
\end{equation}
for each $k\geq0$. Now let $k\geq1$ be arbitrary. Fix a regular
point $t_{k}\in(N_{\widetilde{r},k-1},N_{\widetilde{r},k})$ of the
stopping time $\tau_{x,a,N(\widetilde{r},k)}$. Define the measurable
set
\[
A_{k}\equiv(\tau_{x,a,N(\widetilde{r},k)}\leq t_{k})\equiv(\mathcal{\tau}_{f,a,N(\widetilde{r},k)}(B^{x})\leq t_{k}).
\]
Then, by inequality \ref{eq:temp-11}, we have
\[
P(A_{k}^{c})=P(\tau_{x,a,N(\widetilde{r},k)}>t_{k})\leq2\Phi_{0,1}(\frac{\overline{r}+\widetilde{r}}{\sqrt{t(k)}})-1
\]
\[
\leq2\Phi_{0,1}(\frac{\overline{r}+\widetilde{r}}{\sqrt{N(\widetilde{r},k-1)}})-1.
\]
Letting $\overline{r}\downarrow0$,  we obtain
\begin{equation}
P(A_{k}^{c})\leq2\Phi_{0,1}(\frac{\widetilde{r}}{\sqrt{N(\widetilde{r},k-1)}})-1\leq2^{-k+1},\label{eq:temp-11-1}
\end{equation}
where the last inequality is by inequality \ref{eq:temp-14} applied
to $k-1$. Moreover, for each $n\geq N_{\widetilde{r},k}$, we have
\[
A_{k}\equiv(\tau_{x,a,N(\widetilde{r},k)}\leq t_{k})
\]
\[
\subset(\tau_{x,a,n}=\tau_{x,a,N(\widetilde{r},k)}\leq t_{k})\subset(\tau_{x,a,N(\widetilde{r},k)}=\tau_{x,a,n}),
\]
where the first inclusion relation is due to relation $\ref{eq:temp-17}$.
Thus
\begin{equation}
A_{k}\subset\bigcap_{n=N(\widetilde{r},k)}^{\infty}(\tau_{x,a,N(\widetilde{r},k)}=\tau_{x,a,n}),\label{eq:temp-15-2}
\end{equation}
where $k\geq1$ is arbitrary. Since $P(A_{k}^{c})\leq2^{-k+1}$ is
arbitrarily small as $k\rightarrow\infty$, we see that $\tau_{x,a,n}\rightarrow\tau_{x,a}$
a.u. as $k\rightarrow\infty$, whence $\tau_{x,a}$ is a r.r.v. Furthermore,
relation \ref{eq:temp-15-2} implies
\begin{equation}
A_{k}\subset\bigcap_{n=N(\widetilde{r},k)}^{\infty}(\tau_{x,a,n}=\tau_{x,a})\subset(\tau_{x,a,N(\widetilde{r},k)}=\tau_{x,a}).\label{eq:temp-15-2-1}
\end{equation}
Recall here that 
\[
a\in(\overline{r},\widetilde{r})H_{c}
\]
is arbitrary. 

7. To prove that the r.r.v. $\tau_{x,a}$ is a stopping time, take
an arbitrary regular point $t\in(0,\infty)$ of $\tau_{x,a}$. Let
$(s_{j})_{j=1,2,\cdots}$ be an arbitrary decreasing sequence in $(t,\infty)$
such that $s_{j}\downarrow t$ and such that $s_{j}$ is a regular
point of the r.r.v.'s $\tau_{x,a}$ and $\tau_{x,a,N(\widetilde{r},k)}$
for each $k\geq1$. Consider each $j\geq1$. Then, in view of relation
\ref{eq:temp-15-2-1}, we have
\[
(\bigcup_{k=1}^{\infty}A_{k})(\tau_{x,a}\leq s_{j})=\bigcup_{k=1}^{\infty}A_{k}(\tau_{x,a,N(\widetilde{r},k)}\leq s_{j})
\]
\begin{equation}
=(\bigcup_{k=1}^{\infty}A_{k})(\bigcup_{k=1}^{\infty}(\tau_{x,a,N(\widetilde{r},k)}\leq s_{j}))\label{eq:temp-15}
\end{equation}
Note that $A\equiv\bigcup_{k=1}^{\infty}A_{k}$ is a full set. Hence
$A\in L^{(s(j))}$. At the same time $\bigcup_{k=1}^{\infty}(\tau_{x,a,N(\widetilde{r},k)}\leq s_{j})\in L^{(s(j))}$
because $\tau_{x,a,N(\widetilde{r},k)}$ is a stopping time relative
to the filtration $\mathcal{L}$, for each $k\geq1$. Hence equality
\ref{eq:temp-15} implies that $(\tau_{x,a}\leq s_{j})\in L^{(s(j))},$
where $j\geq1$ is arbitrary. It follows that 
\[
(\tau_{x,a}\leq t)=\bigcap_{j=1}^{\infty}(\tau_{x,a}\leq s_{j})\in L^{(t+)}=L^{(t)},
\]
where the last equality is thanks to the assumed right continuity
of the filtration $\mathcal{L}$. Since $t$ is an arbitrary regular
point of $\tau_{x,a}$, the r.r.v. $\tau_{x,a}$ is a stopping time
relative to $\mathcal{L}$, as alleged.

8. Proceed to prove that $\tau_{x,a}$ is the first exit time $\tau_{D(0,a)}(B^{x})$
of the open ball $D_{0,a}$ by the Brownian motion $B^{x}$. To that
end, recall that $\widetilde{r}>\overline{r}>0$ and $a\in(\overline{r},\widetilde{r})H_{c}$
are arbitrary. Write for abbreviation $\tau\equiv\tau_{x,a}$. Let
$g\in C(\partial D_{0,a})$ be arbitrary and let $a'\equiv2^{-1}a$.
Define $\overline{g}\in C_{ub}(\overline{D}_{0,a',a})$ by 
\begin{equation}
\overline{g}(x)\equiv g(y+\frac{a(x-y)}{\left\Vert x-y\right\Vert })\label{eq:temp-20}
\end{equation}
for each $x\in\overline{D}_{0,a',a}$. Then, for each $k\geq1$, relation
\ref{eq:temp-15-2-1} implies that
\begin{equation}
A_{k}\subset(\tau_{x,a,N(\widetilde{r},k)}=\tau).\label{eq:temp-15-2-1-1}
\end{equation}
At the same time, $B_{\tau(x,a,N(\widetilde{r},k))}^{x}:\Omega\rightarrow R^{m}$
is a well defined r.v. according to \ref{eq:temp-54} in Step 1. Hence
$\overline{g}(B_{\tau(x,a,N(\widetilde{r},k))}^{x})\in L$. Therefore
relation \ref{eq:temp-15-2-1-1} implies that
\begin{equation}
\overline{g}(B_{\tau}^{x})1_{A(k)}=\overline{g}(B_{\tau(x,a,N(\widetilde{r},k))}^{x})1_{A(k)}\in L.\label{eq:temp-19}
\end{equation}
Since $P(A_{k})\uparrow1$ as $k\rightarrow\infty$, the Monotone
Convergence Theorem implies that
\begin{equation}
\overline{g}(B_{\tau}^{x})\in L.\label{eq:temp-21}
\end{equation}

Next let $\omega\in domain(\tau)\cap A$ be arbitrary. Take $k\geq1$
so large that $\tau(\omega)<N_{\widetilde{r},k}$ and $\omega\in A_{k}$.
Then, according to relation \ref{eq:temp-15-2-1-1}, we have $\tau_{x,a,N(\widetilde{r},k)}(\omega)=\tau(\omega)<N_{\widetilde{r},k}$.
Hence, according to Condition (ii') in Step 3, we have 
\[
B_{\tau}^{x}(\omega)=B_{\tau(x,a,N(\widetilde{r},k))}^{x}(\omega)\in\partial D_{0,a}.
\]
Therefore, from the defining equality \ref{eq:temp-20}, we obtain
$g(B_{\tau}^{x})=\overline{g}(B_{\tau}^{x})$ on the full set $domain(\tau)\cap A$.
Consequently, relation \ref{eq:temp-21} implies that $g(B_{\tau}^{x})\in L.$
Since $g\in C(\partial D_{0,a})$ is arbitrary, we conclude that $B_{\tau}^{x}:(\Omega,L,E)\rightarrow\partial D_{0,a}$
is a r.v.

Summing up, Conditions (i-iii) in Definition \ref{Def First exit time from open ball by Brownian motion}
are satisfied with $a$ in the place of $r$. Accordingly, the stopping
time $\tau\equiv\tau_{x,a}$ is the first exit time for the process
$B^{x}$ to exit the open ball $D_{0,a}$, with exit point $B_{\tau}^{x}$.
Assertion 1 is proved.

9. To prove the remaining Assertion 2, let $k\geq0$ be arbitrary.
Fix a regular point $t_{k}\in(N_{\widetilde{r},k-1},N_{\widetilde{r},k})$
of the stopping time $\tau_{x,a,N(\widetilde{r},k)}$. Then 
\[
(\tau_{x,a}>t)\subset(\tau_{x,a}>N_{\widetilde{r},k})\subset(\tau_{x,a}>N_{\widetilde{r},k})A_{k}\cup A_{k}^{c}
\]
\[
\subset(\tau_{x,a,N(\widetilde{r},k)}>N_{\widetilde{r},k})A_{k}\cup A_{k}^{c}
\]
\begin{equation}
=\phi\cup A_{k}^{c}=A_{k}^{c}\label{eq:temp-43}
\end{equation}
where the third inequality is by relation \ref{eq:temp-15-2-1}, and
where the first equality is because the first exit time $\tau_{x,a,N(\widetilde{r},k)}$
has values in $(0,N_{\widetilde{r},k}]$. Hence
\[
P(\tau_{x,a}>t)\leq P(A_{k}^{c})\leq2^{-k+1}
\]
where the second inequality is by inequality \ref{eq:temp-11-1}.
Note here that $\widetilde{r}>\overline{r}>0$ and $a\in(\overline{r},\widetilde{r})H_{c}$
are arbitrary. Assertion 2 and the lemma are proved. 
\end{proof}
The next theorem strengthens the preceding Lemma \ref{Lem.  Existence of certain first exit tims }
by removing the exceptional set $H$ and by proving that the first
hitting time $\tau_{x,a}$ is then an a.u. continuous process with
$a\in(\overline{r},\widetilde{r})$ as parameter. 
\begin{thm}
\textbf{\emph{\label{Thm.  Existence and continuity of first exit time}
Existence and continuity of first exit time.}} Let $\widetilde{r}>\overline{r}>0$
and $x\in D_{0,\overline{r}}$ and be arbitrary. Then the following
conditions hold.

\emph{1}. \emph{(a.u. Continuity on dense subset of the interval $(\overline{r},\widetilde{r})$)}
By Lemma \ref{Lem.  Existence of certain first exit tims }, there
exists a countable subset $H$ of $R$ such that, for each $a\in(\overline{r},\widetilde{r})H_{c},$
the first exit time $\tau_{x,a}\equiv\tau_{D(0,a)}(B^{x})$ exists. 

Let $k\geq1$ be arbitrary. Let $a'_{k},a''_{k}\in(\overline{r},\widetilde{r})H_{c}$
be such that 
\begin{equation}
a''_{k}-a'_{k}<2^{-k-1}\label{eq:temp-55}
\end{equation}
and such that
\begin{equation}
2\Phi_{0,1}(\frac{a''_{k}-a'_{k}}{\sqrt{2^{-k-1}}}))-1<2^{-k-1}\label{eq:temp-23}
\end{equation}
Then there exists a measurable set $A_{k}$ with $P(A_{k}^{c})<2^{-k}$
such that 
\begin{equation}
A_{k}\subset(\tau_{x,a''(k)}-\tau_{x,a'(k)}\leq2^{-k}).\label{eq:temp-40}
\end{equation}

2. \emph{(Existence)}. Let $r\in(\overline{r},\widetilde{r})$ be
arbitrary. Then the first exit time $\tau_{x,r}\equiv\tau_{D(0,r)}(B^{x})$
exists. 

3. \emph{(a.u. Continuity on $(\overline{r},\widetilde{r})$)}. Let
$\kappa\geq1$ be arbitrary. Let $r',r''\in(\overline{r},\widetilde{r})$
be arbitrary with 
\begin{equation}
0\leq r''-r'<2^{-\kappa-1}\label{eq:temp-57}
\end{equation}
and 
\begin{equation}
2\Phi_{0,1}(\frac{r''-r'}{\sqrt{2^{-\kappa-1}}}))-1<2^{-\kappa-1}.\label{eq:temp-23-4}
\end{equation}
Then there exists a measurable set $G_{\kappa}$ with $P(G_{\kappa}^{c})<2^{-\kappa+1}$
such that 
\[
G_{\kappa}\subset(0\leq\tau_{x,r''}-\tau_{x,r'}\leq2^{-\kappa+4}).
\]

4. \emph{(Tightness).} The family $\{\tau_{x,r}:x\in D_{0,\overline{r}};r\in(\overline{r},\widetilde{r})\}$
of r.r.v.'s is tight. Specifically. consider each $r\in(\overline{r},\widetilde{r}).$
Let $k\geq1$ be arbitrary and let $N_{\widetilde{r},k}$ be the integer
defined in \ref{eq:temp-14} of Definition \ref{Def. Modulus of continuity at 0 of standard normal CDF}.
Then 
\[
P(\tau_{x,r}>t)\leq2^{-k+1}
\]
for each $t>N_{\widetilde{r},k}$, 
\end{thm}

\begin{proof}
1. To prove Assertion 1, let $k\geq1$ and the points $a'_{k},a''_{k}\in(\overline{r},\widetilde{r})H_{c}$be
as given, satisfying inequalities \ref{eq:temp-55} and \ref{eq:temp-23}.
For abbreviation, write $\alpha'\equiv\alpha'_{k}$, $\tau'\equiv\tau_{x,a'(k)}$,
$\alpha''\equiv\alpha''_{k}$, $\tau''\equiv\tau_{x,a''(k)}$. Fix
an arbitrary $\delta_{k}\in(2^{-k-1},2^{-k}).$ 

2. We will verify that the measurable set 
\begin{equation}
A_{k}\equiv(\tau_{x,a''(k)}\leq\tau_{x,a'(k)}+\delta_{k}).\label{eq:temp-28}
\end{equation}
has the desired properties in Assertion 1. Because $\delta_{k}<2^{-k}$,
the desired relation \ref{eq:temp-40} is trivial. It remains to show
that 
\begin{equation}
P(A_{k}^{c})=P(\tau_{x,a''(k)}>\tau_{x,a'(k)}+\delta_{k})<2^{-k}.\label{eq:temp-22}
\end{equation}

3. To that end, consider each $\omega\in A_{k}^{c}$ and $s\in[0,\delta_{k}]$.
We can apply condition (iii) in Definition \ref{Def First exit time from open ball by Brownian motion}
to the time point $t\equiv\tau'(\omega)+s$ and the first exit time
$\tau''\equiv\tau_{x,a''(k)}$, to obtain $B_{\tau'+s}^{x}(\omega)\in D_{0,a''(k)}$.
Thus 
\[
A_{k}^{c}\subset\bigcap_{s\in[0,\delta(k)]}(B_{\tau'+s}^{x}\in D_{0,a''(k)})
\]
\begin{equation}
\subset(\bigvee_{s\in[0,\delta(k)]}\left\Vert B_{\tau'+s}^{x}\right\Vert <a_{k}'').\label{eq:temp-22-1}
\end{equation}

4. We will prove that probability of the measurable set on the right-hand
side of relation \ref{eq:temp-22-1} is bounded by $2^{-k}$. For
that purpose, consider the process 
\[
B_{\tau'+\cdot}^{x}:[0,\infty)\times\Omega\rightarrow R^{m}
\]
Applying condition (ii) in Definition \ref{Def First exit time from open ball by Brownian motion}
to the first exit time $\tau'\equiv\tau_{x,a'(k)}$, we see that $B_{\tau'}^{x}$
is a r.v. with values in $\partial D_{0,a'(k)'}$. Hence the r.v.
\begin{equation}
U\equiv a_{k}'^{-1}B_{\tau'}^{x}\label{eq:temp-38}
\end{equation}
has values in $\partial D_{0,1}.$ In term of components, 
\[
U\equiv(U_{1},\cdots,U_{m})\equiv a'^{-1}(B_{1,\tau'}^{x},\cdots,B_{m,\tau'}^{x})
\]
 where $\sum_{j=1}^{m}U_{j}^{2}=1$ and where $U_{j}\in L^{(\tau')}$
for each $j=1,\cdots,m$.

5. Define the process $\widehat{B}:[0,\infty)\times\Omega\rightarrow R^{1}$
by
\[
\widehat{B}_{s}\equiv U\cdot(B_{\tau'+s}^{x}-B_{\tau'}^{x})=\sum_{j=1}^{m}U_{j}(B_{j,\tau'+s}^{x}-B_{j,\tau'}^{x})
\]
\begin{equation}
=U\cdot(B_{\tau'+s}-B_{\tau'})=\sum_{j=1}^{m}U_{j}(B_{j,\tau'+s}-B_{j,\tau'})\label{eq:temp-24}
\end{equation}
for each $s\geq0$, where the dot between two vectors signifies inner
product. Because the process $B^{x}$ is a.u. continuous on compact
subsets of the parameter set $[0,\infty)$, so is the process $\widehat{B}$. 

6. Trivially, $\widehat{B}_{0}=0$. Let the nondecreasing sequence
$(s_{1},\cdots,s_{k})$ in $[0,\infty)$ and the sequence $(\lambda_{1},\cdots,\lambda_{k})$
in $R^{m}$ be arbitrary. Write $i\equiv\sqrt{-1}$ and $s_{0}\equiv0$.
Then the characteristic function 
\[
E(\exp\sum_{k=1}^{n}i\lambda_{k}(\widehat{B}_{s(k)}-\widehat{B}_{s(k-1)}))=E(\prod_{k=1}^{n}\exp i\lambda_{k}(\widehat{B}_{s(k)}-\widehat{B}_{s(k-1)}))
\]
\[
=E(E(\prod_{k=1}^{n}\exp i\lambda_{k}(\widehat{B}_{s(k)}-\widehat{B}_{s(k-1)}))|U_{1},\cdots,U_{m}))
\]
\[
=E(E(\prod_{k=1}^{n}\exp i\lambda_{k}(\sum_{j=1}^{m}U_{j}(B_{j,\tau'+s(k)}-B_{j,\tau'+s(k-1)}))|U_{1},\cdots,U_{m}))
\]
\[
=E(\prod_{k=1}^{n}\prod_{j=1}^{m}\exp i\lambda_{k}U_{j}(\widetilde{B}_{j,s(k)}^{B_{\tau'}^{x}}-\widetilde{B}_{j,s(k-1)}^{B_{\tau'}^{x}}))
\]
\[
=E(\prod_{k=1}^{n}\prod_{j=1}^{m}\exp i\lambda_{k}U_{j}(\widetilde{B}_{j,s(k)}-\widetilde{B}_{j,s(k-1)}))
\]
\[
=E(\prod_{k=1}^{n}\prod_{j=1}^{m}\exp-2^{-1}\lambda_{k}^{2}U_{j}^{2}(s_{k}-s_{k-1}))
\]
\[
=E(\prod_{k=1}^{n}\exp-2^{-1}\lambda_{k}^{2}\sum_{j=1}^{m}U_{j}^{2}(s_{k}-s_{k-1}))
\]
\[
=E(\prod_{k=1}^{n}\exp-2^{-1}\lambda_{k}^{2}(s_{k}-s_{k-1}))
\]
\[
=E(\exp\sum_{k=1}^{n}i\lambda_{k}(B_{1,s(k)}-B_{1,s(k-1)})),
\]
which is the characteristic function of $B_{1,s(1)}-B_{1,s(0)},\cdots,B_{1,s(n)}-B_{1,s(n-1)}$.
Thus the process $\widehat{B}$ is equivalent to the Brownian motion
$B_{1,\cdot}$ in $R^{1}$. Being also a.u. continuous, $\widehat{B}$
is itself a Brownian motion in $R^{1}$.

7. Recall from the defining equality \ref{eq:temp-28} in Step 2 the
measurable set 
\[
A_{k}\equiv(\tau_{x,a''(k)}-\tau_{x,a'(k)}\leq\delta_{k}).
\]
We will verified that its complement has probability bounded by $2^{-k-1}$.
Consider each $s\in[0,\delta_{k}]$ and
\[
\omega\in A_{k}^{c}=(\tau_{x,a''(k)}>\tau'+\delta_{k})\subset(\tau_{x,a''(k)}>\tau'+s).
\]
Then $t\equiv\tau'(\omega)+s<\tau_{x,a''(k)}(\omega)$, Therefore,
$B_{t}^{x}(\omega)\in D_{0,a''(k)}$. Consequently,
\[
\widehat{B}_{s}(\omega)\equiv U(\omega)\cdot(B_{\tau'+s}^{x}(\omega)-B_{\tau'}^{x}(\omega))
\]
\[
\equiv a'^{-1}B_{\tau'}^{x}(\omega)\cdot B_{t}^{x}(\omega)-U(\omega)\cdot B_{\tau'}^{x}(\omega)
\]
\[
\leq a'^{-1}\left\Vert B_{\tau'}^{x}(\omega)\right\Vert \cdot\left\Vert B_{t}^{x}(\omega)\right\Vert -U(\omega)\cdot B_{\tau'}^{x}(\omega)
\]
\[
<a'^{-1}a'a_{k}''-U(\omega)\cdot B_{\tau'}^{x}(\omega)
\]
\[
\equiv a_{k}''-U(\omega)\cdot(a_{k}'U(\omega))=a_{k}''-a'_{k}
\]
where the first inequality is by the Cauchy-Schwarz inequality, and
where the next to last inequality is by the defining equality \ref{eq:temp-38}.
Combining, we obtain
\[
A_{k}^{c}\subset(\bigvee_{s\in[0,\delta(k)]}\widehat{B}_{s}\leq a_{k}''-a'_{k})
\]
Hence
\[
P(A_{k}^{c})\leq P(\bigvee_{s\in[0,\delta(k)]}\widehat{B}_{s}\leq a_{k}''-a_{k})
\]
\[
=1-P(\bigvee_{s\in[0,\delta(k)]}\widehat{B}_{s}>a''_{k}-a'_{k})
\]
\[
=2\Phi_{0,1}(\frac{a''_{k}-a'_{k}}{\sqrt{\delta_{k}}}))-1
\]
\begin{equation}
<2\Phi_{0,1}(\frac{a''_{k}-a'_{k}}{\sqrt{2^{-k-1}}}))-1<2^{-k-1}.\label{eq:temp-23-1}
\end{equation}
where the last equality by applying the reflection principle, Lemma
\ref{Lem. Reflectioin principle}, to the Brownian motion $\widehat{B}$,
and where the last inequality is by inequality \ref{eq:temp-23}.
Inequality \ref{eq:temp-22} and Assertion 1 are proved.

8. Next, let $r\in(\overline{r},\widetilde{r})$ be arbitrary. Let
$(a'_{k})_{k=1,2,\cdots}$be an increasing sequence in $(\overline{r},r)H_{c}$
and let $(a''_{k})_{k=1,2,\cdots}$be a decreasing sequence in $(\overline{r},r)H_{c}$
such that inequalities \ref{eq:temp-55} and \ref{eq:temp-23} hold.
For each $k\geq1$, fix an arbitrary $\delta_{k}\in(2^{-k-1},2^{-k}).$and
define, as in equality \ref{eq:temp-28}, the measurable set 
\begin{equation}
A_{k}\equiv(\tau_{x,a''(k)}\leq\tau_{x,a'(k)}+\delta_{k}).\label{eq:temp-28-1}
\end{equation}
 Then $P(A_{k}^{c})<2^{-k}$ and 
\begin{equation}
A_{k}\subset(\tau_{x,a''(k)}-\tau_{x,a'(k)}\leq2^{-k}).\label{eq:temp-40-2}
\end{equation}

9. Let $\kappa\geq1$ be arbitrary. Define $A_{\kappa+}\equiv\bigcap_{k=\kappa}^{\infty}A_{k}.$
From relation \ref{eq:temp-28-1}, we see that 
\[
A_{\kappa+}\subset\bigcap_{k=\kappa}^{\infty}(\tau_{x,a'(k)}\leq\tau_{x,a'(k+1)}\leq\tau_{x,a''(k+1)}\leq\tau_{x,a''(k)}\leq\tau_{x,a'(k)}+\delta_{k})
\]
\[
\subset\bigcap_{k=\kappa}^{\infty}(\tau_{x,a'(k+1)}-\tau_{x,a'(k)}\leq2^{-k+1})
\]
\[
\cap\bigcap_{k=\kappa}^{\infty}(\tau_{x,a''(k)}-\tau_{x,a''(k+1)}\leq2^{-k+1})
\]
\begin{equation}
\cap\bigcap_{k=\kappa}^{\infty}(\tau_{x,a''(k)}-\tau_{x,a'(k)}\leq2^{-k+1}).\label{eq:temp-39}
\end{equation}
Since $P(A_{\kappa+}^{c})\leq\sum_{k=\kappa}^{\infty}P(A_{k}^{c})<2^{-\kappa}$
is arbitrarily small for sufficiently large $\kappa\geq1$, it follows
that $\tau_{x,a'(k)}\uparrow\tau'$ a.u. for some r.r.v. $\tau'$,
that $\tau_{x,a''(k)}\downarrow\tau''$ for some r.r.v. $\tau''$,
and that $\tau'=\tau''$. Moreover, relation \ref{eq:temp-39} implies
that
\[
A_{\kappa+}\subset(\tau'-\tau_{x,a'(\kappa)}\leq2^{-\kappa+2})
\]
\begin{equation}
\cap(\tau_{x,a''(\kappa)}-\tau''\leq2^{-\kappa+2})\cap(\tau'=\tau'')\label{eq:temp-39-1}
\end{equation}
where $\kappa\geq1$ is arbitrary. 

10. Now define the r.r.v. $\tau:\Omega\rightarrow(0,\infty)$ by 
\[
domain(\tau)\equiv\bigcup_{\kappa=1}^{\infty}A_{\kappa+}^{c}
\]
and
\[
\tau(\omega)\equiv\tau'(\omega)\equiv\lim_{k\rightarrow\infty}\tau_{x,a''(k)}(\omega)\equiv\tau''(\omega)\equiv\lim_{k\rightarrow\infty}\tau_{x,a'(k)}(\omega)
\]
for each $\omega\in domain(\tau)$. By Step 9, we have $\tau_{x,a'(k)}\uparrow\tau$
a.u. and $\tau_{x,a''(k)}\downarrow\tau$ a.u. Since $\tau_{x,a''(k)}$
is a stopping time relative to the right continuous filtration $\mathcal{L}$,
it follows that the a.u. limit $\tau$ is a stopping time relative
to $\mathcal{L}$.

11. Proceed to verify that $\tau$ is the first exit time $\tau_{D(0,r)}(B^{x})$.
To that end, consider each $\omega\in domain(\tau)$. Then $\tau_{x,a''(k)}(\omega)\downarrow\tau(\omega)$
as $k\rightarrow\infty$. Hence 
\[
B_{\mathcal{\tau}}^{x}(\omega)=\lim_{k\rightarrow\infty}B_{\mathcal{\tau}(x,a''(k))}^{x}(\omega)\in\bigcap_{k=1}^{\infty}\overline{D}_{0,r,a''(k)}=\partial D_{0,r}.
\]
Thus condition (i) of Definition \ref{Def First exit time from open ball by Brownian motion}
for the first exit time has been proved for $\tau$. Next, let $g\in C(\partial D_{0,r})$
be arbitrary. Define the function $\overline{g}\in C_{ub}(\overline{D}_{0,\overline{r},\widetilde{r}})$
by 
\begin{equation}
\overline{g}(z)\equiv g(\frac{rz}{\left\Vert z\right\Vert })\label{eq:temp-20-1}
\end{equation}
for each $z\in\overline{D}_{0,\overline{r},\widetilde{r}}$. Since
$\mathcal{\tau}_{x,a''(k)}\downarrow\tau$ a.u. as $k\rightarrow\infty.$,
we have $\overline{g}(B_{\mathcal{\tau}(x,a''(k))}^{x})\rightarrow\overline{g}(B_{\mathcal{\tau}}^{x})$
a.u. Therefore $\overline{g}(B_{\mathcal{\tau}}^{x})$ is integrable.
At the same time, by the defining equality \ref{eq:temp-20-1}, we
see that $g(B_{\mathcal{\tau}}^{x})=\overline{g}(B_{\mathcal{\tau}}^{x})$.
Hence $g(B_{\mathcal{\tau}}^{x})$ is integrable, with
\begin{equation}
Eg(B_{\mathcal{\tau}}^{x})=E\overline{g}(B_{\mathcal{\tau}}^{x})\label{eq:temp-25}
\end{equation}
where $g\in C(\partial D_{0,r})$ is arbitrary. Thus $B_{\mathcal{\tau}}^{x}:\Omega\rightarrow\partial D_{0,r}$
is a r.v. Condition (ii) of Definition \ref{Def First exit time from open ball by Brownian motion}
is verified for the stopping time $\tau$. Finally, consider each
$\omega\in domain(\tau)$ and $t\in[0,\tau(\omega))$. Then $t\in[0,\mathcal{\tau}_{x,a'(k)}(\omega))$
for some $k\geq1$. Hence $B_{t}^{x}(\omega)\in D_{0,a'(k)}\subset D_{0,r}.$
All three conditions of Definition \ref{Def First exit time from open ball by Brownian motion}
have been verified for the stopping time $\tau$. Accordingly, $\tau$
is the first exit time $\tau_{x,r}\equiv\tau_{D(0,r)}(B^{x})$. Thus
Assertion 2 is proved. Moreover, relation \ref{eq:temp-39-1} implies
that
\[
A_{\kappa+}\subset(\tau_{x,r}-\tau_{x,a'(\kappa)}\leq2^{-\kappa+2})
\]
\begin{equation}
\cap(\tau_{x,a''(\kappa)}-\tau_{x,r}\leq2^{-\kappa+2}),\label{eq:temp-39-1-2}
\end{equation}
where $P(A_{\kappa+}^{c})<2^{-\kappa}$, and where $\kappa\geq1$
is arbitrary. 

12. We will next prove Assertion 3. To that end, let $r',r''\in(\overline{r},\widetilde{r})$
be arbitrary with $r'\leq r''$. Suppose 
\[
r''-r'<2^{-\kappa-1}
\]
and 
\begin{equation}
2\Phi_{0,1}(\frac{r''-r'}{\sqrt{2^{-\kappa-1}}}))-1<2^{-\kappa-1}\label{eq:temp-23-4-1}
\end{equation}
for some $\kappa\geq1$. Then there exist $a'_{\kappa},a''_{\kappa}\in(\overline{r},r')H_{c}$
such that $a''_{\kappa}-a'_{\kappa}<2^{-\kappa-1}$ and such that
$a'_{\kappa}<r'\leq r''<a''_{\kappa}$ with $a''_{\kappa}-a'_{\kappa}<2^{-\kappa-1}$and
\begin{equation}
2\Phi_{0,1}(\frac{a''_{\kappa}-a'_{\kappa}}{\sqrt{2^{-\kappa-1}}}))-1<2^{-\kappa-1}.\label{eq:temp-23-3}
\end{equation}
As in equality \ref{eq:temp-28} of Step 2, define the measurable
set 
\begin{equation}
A_{\kappa}\equiv(\tau_{x,a''(\kappa)}-\tau_{x,a'(\kappa)}\leq\delta_{\kappa}).\label{eq:temp-40-1}
\end{equation}
Relation \ref{eq:temp-39-1-2} applied to $r\equiv r'$ then implies
that there exists a measurable set $\overline{A}{}_{\kappa}\subset A_{\kappa}$
with $P(\overline{A}{}_{\kappa}^{c})<2^{-\kappa}$, and 
\begin{equation}
\overline{A}{}_{\kappa}\subset(\tau_{x,r'}-\tau_{x,a'(\kappa)}\leq2^{-\kappa+2})\label{eq:temp-39-1-2-1}
\end{equation}
Relation \ref{eq:temp-39-1-2} applied to $r=r''$ similarly implies
that there exists a measurable set $\overline{\overline{A}}{}_{\kappa}\subset A_{\kappa}$
with $P(\overline{\overline{A}}{}_{\kappa}^{c})<2^{-\kappa}$, and
with 
\begin{equation}
\overline{\overline{A}}{}_{\kappa}\subset(\tau_{x,a''(\kappa)}-\tau_{x,r''}\leq2^{-\kappa+2}).\label{eq:temp-39-1-2-1-1}
\end{equation}
Now define the measurable set
\[
G_{\kappa}\equiv\overline{A}{}_{\kappa}\cap\overline{\overline{A}}{}_{\kappa}.
\]
Then inequalities \ref{eq:temp-40-1}, \ref{eq:temp-39-1-2-1}, and
\ref{eq:temp-39-1-2-1-1} together implies that
\[
G_{\kappa}=\overline{A}{}_{\kappa}\cap\overline{\overline{A}}{}_{\kappa}.\cap A_{\kappa}\subset(\tau_{x,r''}-\tau_{x,r'}\leq2^{-\kappa+2}+2^{-\kappa+2}+\delta_{\kappa})
\]
\[
\subset(\tau_{x,r''}-\tau_{x,r'}\leq2^{-\kappa+4}),
\]
where 
\[
P(G_{\kappa}^{c})\equiv P(\overline{A}{}_{\kappa}^{c}\cup\overline{\overline{A}}{}_{\kappa}^{c})<2^{-\kappa}+2^{-\kappa}=2^{-\kappa+1}.
\]
Assertion 3 is proved.

13. It remains to prove that the family $\{\tau_{x,r}:x\in D_{0,\overline{r}};r\in(\overline{r},\widetilde{r})\}$
is tight. To that end, consider each $x\in D_{0,\overline{r}}$ and
$r\in(\overline{r},\widetilde{r}).$ Let $k\geq1$ be arbitrary. Take
any $a_{k}',a_{k}''\in(\overline{r},\widetilde{r})H_{c}$ such that
$a_{k}'<r<a_{k}''$ and such that inequalities \ref{eq:temp-55} and
\ref{eq:temp-23} hold. Then, according to Assertion 1, there exists
a measurable set $A_{k}$ with $P(A_{k}^{c})<2^{-k}$ such that 
\begin{equation}
A_{k}\subset(\tau_{x,a''(k)}-\tau_{x,a'(k)}\leq2^{-k}).\label{eq:temp-40-3-1}
\end{equation}
Consider each $t>N_{\widetilde{r},k}$, where $N_{\widetilde{r},k}$
is the integer defined in \ref{eq:temp-14} of Definition \ref{Def. Modulus of continuity at 0 of standard normal CDF}.
Then, according to inequality \ref{eq:temp-58} in Assertion 2 of
Lemma \ref{Lem.  Existence of certain first exit tims }, we have
\[
P(\tau_{x,a'(k)}>t)\leq2^{-k+1},
\]
 Since $r>a_{k}'$ and $\tau_{x,r}\geq\tau_{x,a'(k)}$, it follows
that
\[
P(\tau_{x,r}>t)\leq2^{-k+1}
\]
where $x\in D_{0,\overline{r}}$, $\widetilde{r}>\overline{r}>0$,
$k\geq1$, and $t>N_{\widetilde{r},k}$ are arbitrary. Thus the family
$\{\tau_{x,r}:x\in D_{0,\overline{r}};r\in(\overline{r},\widetilde{r})\}$
is tight. Assertion 4 is proved. 
\end{proof}
\begin{lem}
\emph{\label{Lem. Translational invarinace of certain first exit times}}\textbf{\emph{
Invariance of first exit time relative to time scaling. }}Let $r>0$
be arbitrary. Then 
\begin{equation}
\tau_{0,\sqrt{r}}(\overline{B})=r\tau_{0,1}(B),\label{eq:temp-29}
\end{equation}
where $\overline{B}$ is the Brownian motion defined by $\overline{B}_{s}\equiv\sqrt{r}B_{s/r}$
for each $s\geq0$.
\end{lem}

\begin{proof}
Consider each $\omega\in\Omega$. Write $t\equiv\tau_{0,\sqrt{r}}(\overline{B})(\omega)$.
Thus $t$ is the first time when $|\overline{B}_{t}(\omega)|=|\sqrt{r}B_{t/r}(\omega)|=|\sqrt{r}|$.
Hence $t$ is the first time when $|B_{t/r}|=1$. Therefore $t/r=\tau_{0,1}(B)(\omega)$.
In other words, $\tau_{0,1}(B)(\omega)=\tau_{0,\sqrt{r}}(\overline{B})(\omega)/r,$
where $\omega\in\Omega$ is arbitrary. Equality \ref{eq:temp-29}
follows.
\end{proof}
\begin{thm}
\label{Thm. Exit distribution for the (m-1)-sphere} \textbf{\emph{Exit
distribution for the $(m-1)$-sphere.}} Let $x\in D_{0,1}$ and $r\in(0,1]$
be arbitrary such that $x\in D_{0,r}$. Write $\tau_{x,r}\equiv\tau_{D(0,r)}(B^{x})$.
Then the following conditions hold.

1. For each $f\in C(\partial D_{0,1})$ we have 
\begin{equation}
Ef(B_{\tau(x,1)}^{x})=\int_{z\in\partial D(0,1)}\frac{1-\left\Vert x\right\Vert ^{2}}{\left\Vert x-z\right\Vert ^{m}}f(z)\overline{\sigma}_{m,0,1}(dz).\label{eq:temp-30}
\end{equation}

2. In terms of the Poisson kernel $\mathbf{k}_{0,1}$ introduced in
Theorem \ref{Thm.  Poisson kernel and Poisson integration on the m-sphere},
equality \ref{eq:temp-30} can be restated as
\[
Ef(B_{\tau(x,1)}^{x})=\int_{z\in\partial D(0,1)}\mathbf{k}_{0,1}(x,z)f(z)\overline{\sigma}_{m,0,1}(dz).
\]
Thus $B_{\tau(x,1)}^{x}$ induces a distribution on $\partial D_{0,1}$
that has the density function $\mathbf{k}_{0,1}(x,\cdot)$ relative
to the uniform distribution $\overline{\sigma}_{m,0,1}$. 

3. For each $f\in C(\partial D_{0,r})$ we have 
\begin{equation}
Ef(B_{\tau(r)})=\int_{z\in\partial D(0,r)}f(z)\overline{\sigma}_{m,0,r}(dz).\label{eq:temp-30-1}
\end{equation}
In other words, the r.v. $B_{\tau(r)}:\Omega\rightarrow\partial D_{0,r}$
induces the uniform distribution $\overline{\sigma}_{m,0,r}$ on $\partial D_{0,r}$.
\end{thm}

\begin{proof}
1. For the proof of equality \ref{eq:temp-30}, see assertion (1)
in section 1.10 of \cite{Durret84}. Assertion 2 of the present theorem
is then a trivial consequence of the definition of the Poisson kernel. 

2. To prove Assertion 3, let $\alpha$ be an arbitrary rotation matrix.
Then
\[
Ef\circ\alpha(B_{\tau(r;B)})=Ef((\alpha B)_{\tau(r;\alpha B)})=Ef(\widetilde{B}_{\tau(1;\widetilde{B})})=Ef(B_{\tau(1;B)}).
\]
Thus the distribution induced on $C(\partial D_{0,r})$ by the random
exit point $B_{\tau(r)}\equiv B_{\tau(1;B)}$ is invariant relative
to rotation about the center. By Corollary \ref{Cor. distribution on D(y,r) invariant relative to rotation is equal to uniform distribution-1},
we see that $B_{\tau(,r)}$ induces the uniform distribution on $C(\partial D_{0,r})$.
\end{proof}
In the remainder of this article, for abbreviation we will write $D\equiv D_{0,1}\equiv D_{m,0,1}$,
$\overline{D}\equiv\overline{D}_{0,1}\equiv\overline{D}_{m,0,1}$,
and $\partial D\equiv\partial D_{0,1}\equiv\partial D_{m,0,1}$. For
each $r>0$, we will write $\tau_{r}\equiv\tau_{0,r}$ for the first
exit time for the Brownian motion $B$ to exit the open ball $D_{0,r}$. 
\begin{thm}
\emph{\label{Thm. Dynkin's theorem}}\textbf{\emph{ Dynkin's theorem.
}}Let $a>0$ be arbitrary. Let $u$ be an arbitrary continuous function
on $\overline{D}_{0,a}$ which is harmonic on $D_{0,a}$. Define the
process 
\[
X:[0,\infty)\times(\Omega,L,E)\rightarrow R
\]
by 
\[
X_{t}\equiv u(B_{t\wedge\tau(a)})
\]
for each $t\in[0,\infty)$. Then $X$ is a martingale relative to
the right continuous filtration $\mathcal{L}$.
\end{thm}

\begin{proof}
The present theorem of Dynkin for the case where $a=1$ is cited as
assertion (5) on page 26 of \cite{Durret84} along with a proof which
is an application of Ito's lemma in the theory of stochastic integration
relative to the Brownian motion. The reader can convince himself or
herself that both said proof and said theory of stochastic integration
presented in chapter 2 of \cite{Durret84} are constructive, or can
easily be made so. In addition, the reader should convince himself
or herself that the assumption $a=1$ is not essential.
\end{proof}
\begin{lem}
\label{Lem. Exit point of B on boundary of ball induces an integration on the boundary}\emph{
$B_{\tau(r)}$}\textbf{\emph{ induces an integration}}\emph{ $\partial D_{0,r}$}\textbf{\emph{.}}
Let $r>0$ be arbitrary. 

1. Let $g:\partial D\rightarrow R$ be an arbitrary function. Then
$g$ is an integrable function on $\partial D$ relative to $\overline{\sigma}_{m,0,1}$
iff $g\circ(r^{-1}B_{\tau(r)})$ is integrable on $\Omega$ relieve
to $E$, in which case 
\[
\overline{\sigma}_{m,0,1}g=E(g\circ(r^{-1}B_{\tau(r)})).
\]

2. Let $f$ be an arbitrary continuous function on $\partial D_{0,r}$.
Define the function $g:\partial D\rightarrow R$ by
\begin{equation}
{\normalcolor g\equiv f(r\cdot)}.\label{eq:temp-52-1-1}
\end{equation}
Then
\[
\overline{\sigma}_{m,0,1}g=E(f\circ B_{\tau(r)}).
\]
\end{lem}

\begin{proof}
1. Assertion 3 of Theorem \ref{Thm. Exit distribution for the (m-1)-sphere}
says that the random point $B_{\tau(r)}:(\Omega,L,E)\rightarrow\partial D_{0,r}$
of exit induces the uniform distribution $\overline{\sigma}_{m,0,r}$
on $\partial D_{0,r}.$ Hence the r.v. $r^{-1}B_{\tau(r)}:(\Omega,L,E)\rightarrow\partial D$
induces the uniform distribution $\overline{\sigma}_{m,0,1}$ on $\partial D$.
Assertion 1 of the present lemma follows.

2. Next, Let $f$ be an arbitrary continuous function on $\partial D_{0,r}$.
Define the function $g:\partial D\rightarrow R$ by ${\normalcolor g\equiv f(r\cdot)}.$
Then
\[
\overline{\sigma}_{m,0,1}g\equiv\overline{\sigma}_{m,0,1}f(r\cdot)=\overline{\sigma}_{m,0,r}f=E(f\circ B_{\tau(r)}).
\]
\end{proof}
\begin{cor}
\emph{\label{Cor. Observations of a harmonic func on D of BM at successive exit times}}\textbf{\emph{
Observations of a harmonic function of the Brownian motion at successive
first exit times of concentric open balls constitute an a.u. continuous
martingale.}} Let $u$ be an arbitrary harmonic function on $D\equiv D_{0,1}$.
Then the following conditions hold.

\emph{1.} For each $r\in[0,1)$ , we have $u(B_{\tau(s)})\rightarrow Y_{r}$
a.u. as $s\downarrow r$ for some $Y_{r}\in L.$

\emph{2}. For each $r\in(0,1)$, we have $Y_{r}=u(B_{\tau(r)})$ .

\emph{3.} $Y_{0}=u(0)$.

\emph{4}. The process 
\[
Y:[0,1)\times(\Omega,L,E)\rightarrow R
\]
 is an a.u. continuous martingale relative to the right continuous
filtration
\[
\widehat{\mathcal{L}}\equiv\{L^{(\tau(r))}:r\in[0,1)\}.
\]

\emph{5.} The expectations $E|u(B_{\tau(r)})|$ and $\overline{\sigma}_{m,0,1}|u(r\cdot)|$
are equal and nondecreasing in $r\in(0,1)$. 

\emph{6}. The expectations $E\exp(-|u(B_{\tau(r)})|)$ and $\overline{\sigma}_{m,0,1}\exp(-|u(r\cdot)|)$
are equal and nondecreasing in $r\in(0,1)$. 
\end{cor}

\begin{proof}
1. Let $r\in[0,1)$ be arbitrary. Then $r\in(0,a]$ for some $a<1$.
By hypothesis, the function $u$ is continuous on $\overline{D}_{0,a}$
and is harmonic on $D_{0,a}$. Hence Theorem \ref{Thm. Dynkin's theorem}
is applicable, and implies that the process $X:[0,\infty)\times(\Omega,L,E)\rightarrow R,$
defined by $X_{t}\equiv u(B_{t\wedge\tau(a)})$ is a martingale. 

2. Because the a.u. continuous process $B_{\cdot\wedge\tau(a)}$ has
values in the compact ball $\overline{D}_{0,a}$ and because the function
$u$ is continuous on $\overline{D}_{0,a}$, so the martingale $X\equiv u(B_{\cdot\wedge\tau(a)})$
is a.u. continuous and bounded. 

3. Let $\varepsilon>0$ be arbitrary. Then there exists $\overline{t}(\varepsilon)>0$
so large that $P(G)<\varepsilon$ where $G\equiv(\tau_{0,a}>\overline{t}(\varepsilon))$
and
\[
G^{c}\equiv(\tau_{0,a}\leq\overline{t}(\varepsilon)).
\]
 Because $X\equiv u(B_{\cdot\wedge\tau(a)})$ is a.u. continuous on
$[0,\overline{t}(\varepsilon)]$, there exists a measurable set $\overline{G}$
with $P(\overline{G})<\varepsilon$ and $\overline{\delta}(\varepsilon)>0$
such that 
\[
\overline{G}^{c}\subset\bigcap_{t,t'\in[0,\overline{t}];|t-t'|<\overline{\delta}(\varepsilon)}(|X_{t}-X_{t'}|\leq\varepsilon).
\]
Because $\tau_{0,\cdot}$ is a.u. continuous on (0,1{]}, there exists
a measurable set $\widetilde{G}$ with $P(\widetilde{G})<\varepsilon$
and $\widetilde{\delta}(\varepsilon)>0$ such that 
\[
\widetilde{G}^{c}\subset\bigcap_{s,s'\in(0,a];|s-s'|<\widetilde{\delta}(\varepsilon)}(|\tau_{0,s}-\tau_{0,s'}|\leq\overline{\delta}).
\]

Consider each $\omega\in G^{c}\overline{G}^{c}\widetilde{G}^{c}$.
Let $s,s'\in(r,a]$ be arbitrary with $|s-s'|<\widetilde{\delta}(\varepsilon)$.
Then
\[
\tau_{0,s}(\omega)\vee\tau_{0,s'}(\omega)\leq\tau_{0,a}(\omega)\leq\overline{t}
\]
 and $|\tau_{0,s}(\omega)-\tau_{0,s'}(\omega)|<\overline{\delta}$,
whence $|X_{\tau(s)}(\omega)-X_{\tau(s')}(\omega)|\leq\varepsilon$.
Thus 
\[
G^{c}\overline{G}^{c}\widetilde{G}^{c}\subset\bigcap_{s,s'\in(r,a];|s-s'|<\widetilde{\delta}(\varepsilon).}(|X_{\tau(s)}-X_{\tau(s')}|\leq\varepsilon).
\]
Since $P(G\cup\overline{G}\cup\widetilde{G})<3\varepsilon$ is arbitrarily
small, we see that $X_{\tau(s)}\rightarrow Y_{r}$ a.u. for some $Y_{r}\in L$
as $s\downarrow r$. In other words, $u(B_{\tau(s)})\rightarrow Y_{r}$
a.u. as $s\downarrow r$ Assertion 1 is proved.

4. Assertions 2 and 3 are then trivial. 

5. Note that, for each $a>0$ and $r\in(0,a]$, we have 
\[
Y_{r}=u(B_{\tau(r)})=u(B_{\tau(r)\wedge\tau(a)})=X_{\tau(r)}.
\]
Thus the process $Y:(0,a]\times(\Omega,L,E)\rightarrow R$ is obtained
by optional sampling of the a.u. continuous martingale $X$ which
is bounded. Hence Doob's optional sampling theorem implies that the
process $Y$ is a martingale relative to the filtration $\mathcal{\widehat{L}\equiv}\{L^{(\tau(r))}:r\in(0,1)\}$.
(See \cite{Chan21} for Doob's optional sampling theorem). Since the
stopping time $\tau_{0,r}$ is monotone and a.u. continuous in $r$,
the process $Y:[0,1)\times(\Omega,L,E)\rightarrow R$ is an a.u. continuous
martingale relative to the right continuous filtration $\widehat{\mathcal{L}}$.
This proves Assertion 4. 

6. The expectation $E|u(B_{\tau(r)})|=E|Y_{r}|$ is nondecreasing
because $|Y_{r}|$ is a wide sense submartingale. Since $B_{\tau(r)}$
induces the uniform distribution on $\partial D_{0,r},$we have 
\[
E|u(B_{\tau(r)})|=\overline{\sigma}_{m,0,r}|u|=\overline{\sigma}_{m,0,1}|u(r\cdot)|.
\]
Hence $\overline{\sigma}_{m,0,1}|u(r\cdot)|$ is nondecreasing in
$r\in(0,1)$. Assertion 5 is proved.

7. Similarly, the expectation $E\exp(-|u(B_{\tau(r)})|)=E\exp(-|Y_{r}|)$
is nondecreasing because $|Y_{r}|$ is a wide sense submartingale
and because $\exp(-y)$ is a convex function of $y\in R$. Since $B_{\tau(r)}$
induces the uniform distribution on $\partial D_{0,r},$we have 
\[
E\exp(-|u(B_{\tau(r)})|)=\overline{\sigma}_{m,0,r}\exp(-|u|)=\overline{\sigma}_{m,0,1}\exp(-|u(r\cdot)|).
\]
Hence $\overline{\sigma}_{m,0,1}\exp(-|u(r\cdot)|)$ is nondecreasing
in $r\in(0,1)$. Assertion 6 and the lemma are proved.
\end{proof}

\subsection{Maximal inequality for a martingale}
\begin{defn}
\label{Def. Special convex function}\textbf{ The special symmetric
convex function. }Define the continuous function $\overline{\lambda}:R\rightarrow R$
by
\begin{equation}
\overline{\lambda}(v)\equiv e^{-|v|}-1+|v|=\frac{|v|^{2}}{2!}-\frac{|v|^{3}}{3!}\pm\cdots\label{eq:temp-524-1}
\end{equation}
for each $v\in R$. We will call $\overline{\lambda}$ the \index{special convex function}\emph{special
symmetric convex function}. Then $\overline{\lambda}$ is continuously
differentiable and strictly convex on $R$, with 
\begin{equation}
|\overline{\lambda}(v)|\leq|v|\label{eq:temp-549}
\end{equation}
for each $v\in R$. $\square$
\end{defn}

\begin{thm}
\label{Thm. Maximal inequality for a martingale}\textbf{\emph{ Maximal
inequality for a martingale.}} Let $Z:\{0,1,\cdots,n\}\times\Omega\rightarrow R$
be an arbitrary martingale. Let $\varepsilon>0$ be arbitrary. Suppose
\begin{equation}
E\overline{\lambda}(Z_{n})-E\overline{\lambda}(Z_{0})<\frac{1}{6}\varepsilon^{3}\exp(-\frac{3}{\varepsilon}(E|Z_{0)}|\vee E|Z_{n}|)).\label{eq:temp-523}
\end{equation}
Then 
\begin{equation}
P(\bigvee_{k=0}^{n}|Z_{k}-Z_{0}|>\varepsilon)<\varepsilon.\label{eq:temp-509}
\end{equation}
\end{thm}

\begin{proof}
It can easily be verified that $\overline{\lambda}$ satisfies the
defining conditions for admissible function in the maximal inequality
of Chapter 8 in \cite{Bishop67}, of which the present theorem is
therefore a special case. 
\end{proof}

\section{Brownian Limit Theorem for Hardy space }
\begin{defn}
\label{Def. Hardy space} \textbf{Hardy space. }A harmonic function
$u$ on $D$ is said to be a member of the \emph{Hardy space }\index{Hardy space}\emph{
}$\mathbf{h}^{p}$ if there exists $b_{0}>0$ such that 
\begin{equation}
\int_{z\in\partial D}|u(rz)|^{p}\overline{\sigma}_{m,0,1}(dz)<b_{0}\label{eq:temp-443-1-2}
\end{equation}
for each $r\in(0,1)$. Recall here that $\overline{\sigma}_{m,0,r}$
is the uniform distribution on the $(m-1)$--sphere $\partial D_{0,r}$
for each $r>0$. Note that, by Lyapunov's inequality, we have $\mathbf{h}^{p}\subset\mathbf{h}\equiv\mathbf{h}^{1}$.
Without loss of generality, we will assume that $u(0)=0$.
\end{defn}

$\square$
\begin{lem}
\emph{\label{Lem. Monotonicity and boundedness of certain integrals for members of Hardy space}}\textbf{\emph{Monotonicity
and boundedness of certain integrals for each member of the Hardy
space $\mathbf{h}^{p}$. }}Let the dimension $m\geq2$ and the exponent
$p\geq1$ be arbitrary, but fixed. Let $u\in\mathbf{h}^{p}$ be arbitrary.
Then $u\in\mathbf{h}$ and the following conditions hold.

1. For each $r\in(0,1)$, the function $u$ is uniformly continuous
on $\overline{D}_{0,r}$, with some modulus of continuity $\delta_{u,r}$.

2. The integral 
\begin{equation}
I_{1,r}\equiv\overline{\sigma}_{m,0,1}|u(r\cdot)|\equiv\int_{z\in\partial D}|u(rz)|\overline{\sigma}_{m,0,1}(dz)\label{eq:temp-548-1-1-1}
\end{equation}
is a nondecreasing function of $r\in(0,1)$\emph{.} Moreover, there
exists $b_{0}>0$ such that $I_{1,r}\leq b_{0}$ for each $r\in(0,1)$. 

3. The integral 
\begin{equation}
I_{2,r}\equiv\overline{\sigma}_{m,0,1}(e^{-|u(r\cdot)|})\equiv\int_{z\in\partial D}e^{-|u(rz)|}\overline{\sigma}_{m,0,1}(dz)\label{eq:temp-548-1-1-2-1-1}
\end{equation}
 is a nondecreasing function\emph{ }of\emph{ $r\in(0,1)$. }Moreover,
$I_{2,r}\leq1$ for each $r\in(0,1)$. 
\end{lem}

\begin{proof}
1. Assertions 1 and the boundedness of $I_{1,r}$ immediately follows
from Definition \ref{Def. Hardy space} for the Hardy space.

2. According to assertion 5 of Corollary \ref{Cor. Observations of a harmonic func on D of BM at successive exit times},
the integral $I_{1,r}\equiv\overline{\sigma}_{m,0,1}|u(r\cdot)|$
is nondecreasing in $r\in(0,1)$. Since $I_{1,r}$ has just been proved
to be bounded, Assertion 2 of the present lemma follows.

3. According to assertion 6 of Corollary \ref{Cor. Observations of a harmonic func on D of BM at successive exit times},
the integral $I_{2,r}\equiv\overline{\sigma}_{m,0,1}e^{-|u(r\cdot)|}$
is nondecreasing. Since $e^{-|u(r\cdot)|}\leq1$, we see that $I_{2,r}\leq1$.
Assertion 3 of the present lemma is proved.
\end{proof}
In the following, for abbreviation, we will write $\tau_{r}\equiv\tau_{0,r}$
for each $r>0$. If emphasis is needed for the underlying Brownian
motion $B$, then we write $\tau_{r;B}\equiv\tau_{0,r;B}$ for each
$r>0$. Next is the main theorem of the present paper
\begin{thm}
\emph{\label{Thm Brownian limit theorem for Hardy space}}\textbf{\emph{
Brownian limit theorem for the Hardy space $\mathbf{h}$.}} Let the
dimension $m\geq2$ and the exponent $p\geq1$ be arbitrary, but fixed.
Let $u$ be an arbitrary member of the Hardy space $\mathbf{h}^{p}$.
Suppose the following conditions\emph{ (i)} and \emph{(ii)} hold.

\emph{(i) There exists} $b_{1}\geq0$ such that $I_{1,r}\equiv\overline{\sigma}_{m,0,1}|u(r\cdot)|\rightarrow b_{1}$
as $r\rightarrow1$ with $r\in(0,1)$. More precisely, suppose that,
for each $\varepsilon>0$ there exists $\delta_{1}(\varepsilon)\in(0,1)$
so small that 
\begin{equation}
0\leq b_{1}-I_{1,r}<\varepsilon\label{eq:temp-31-2-1}
\end{equation}
for each $r\in(1-\delta_{1}(\varepsilon),1)$. The number $b_{1}$
and the operation $\delta_{1}$ are, respectively, the limit and the
\emph{rate of convergence}, of the integral $I_{1,r}$ as $r\rightarrow1$
with $r\in(0,1)$. By replacing $\delta_{1}$ if necessary, we may,
without loss of generality, assume that $\delta_{1}(\varepsilon)\downarrow0$
as $\varepsilon\downarrow0$. 

\emph{(ii) There exists} $b_{2}\geq0$ such that $I_{2,r}\equiv\overline{\sigma}_{m,0,1}(e^{-|u(r\cdot)|})\rightarrow b_{2}$
as $r\rightarrow1$ with $r\in(0,1)$. More precisely, suppose that,
for each $\varepsilon>0$, there exists $\delta_{2}(\varepsilon)\in(0,1)$
so small that 
\begin{equation}
0\leq b_{2}-I_{2,r}<\varepsilon\label{eq:temp-31-2}
\end{equation}
for each $r\in(1-\delta_{2}(\varepsilon),1)$. The number $b_{2}$
and the operation $\delta_{2}$ are, respectively, the limit and the
\emph{rate of convergence}, of the integral $I_{2,r}$ as $r\rightarrow1$
with $r\in(0,1)$. By replacing $\delta_{2}$ if necessary, we may,
without loss of generality, assume that $\delta_{2}(\varepsilon)\downarrow0$
as $\varepsilon\downarrow0$.

Then there exist an increasing sequence $(r_{q})_{q=1,2,\cdots}$
in $(0,1)$ with $r_{q}\uparrow1$ and a r.r.v. $V$ such that, for
each $q\geq1$, there exists a measurable set $\overline{A}_{q}$
with $P(\overline{A}_{q})<2^{-q+4}$ and
\begin{equation}
\overline{A}_{q}^{c}\subset(\bigvee_{s\in[\tau(r(q)),\tau(1))}|V-u(B_{s})|\leq2^{-q+3}).\label{eq:temp-34-1-1-2-1}
\end{equation}
In short, as $s\uparrow1$, we have $u(B_{s})\rightarrow V$ a.u.
\end{thm}

\begin{proof}
1. We will verified that the integral 
\begin{equation}
I_{3,r}\equiv\overline{\sigma}_{m,0,1}\overline{\lambda}(u(r\cdot))\label{eq:temp-548-1}
\end{equation}
is a nondecreasing function of\emph{ $r\in(0,1)$. }To that end, we
will first show that
\begin{equation}
I_{3,r}\uparrow\gamma\equiv b_{2}-1+b_{1}\label{eq:temp-37}
\end{equation}
as $r\rightarrow1$ with $r\in(0,1)$. More precisely, we will prove
that for each $r\in(0,1)$ and $\varepsilon>0$, we have 
\begin{equation}
0\leq\gamma-I_{3,r}<\varepsilon\label{eq:temp-48}
\end{equation}
provided that
\[
1-r<\delta_{3}(\varepsilon)\equiv\delta_{1}(2^{-1}\varepsilon)\wedge\delta_{2}(2^{-1}\varepsilon).
\]
Thus $\gamma$ and the function $\delta_{3}$ are, respectively, the
limit and a rate of convergence of the integral $I_{3,r}$ as $r\rightarrow1.$
Note that both $\gamma$ and $\delta_{3}$ depend only on the given
items $b_{1},b_{2},\delta_{1},\delta_{2}$.

2. Note that 
\[
I_{3,r}\equiv\overline{\sigma}_{m,0,1}\overline{\lambda}(u(r\cdot))\equiv\int_{z\in\partial D}\overline{\lambda}(u(rz))\overline{\sigma}_{m,0,1}(dz)
\]
\[
=\int_{x\in\partial D(0,r)}\overline{\sigma}_{m,0,r}(dx)\overline{\lambda}(u(x))
\]
\begin{equation}
=E\overline{\lambda}(u(B_{\tau(r)}))\equiv E\overline{\lambda}(Y_{r}),\label{eq:temp-32}
\end{equation}
where the third equality is by an application of Corollary \ref{Cor. distribution on D(y,r) invariant relative to rotation is equal to uniform distribution-1},
where fourth equality is because the r.v. $B_{\tau(r)}$ induces the
uniform distribution $\overline{\sigma}_{m,0,r}$ on $\partial D_{0,r}$
according to Assertion 3 of Theorem \ref{Thm. Exit distribution for the (m-1)-sphere},
and where the last equality is by Assertion 2 of Corollary \ref{Cor. Observations of a harmonic func on D of BM at successive exit times}
regarding the martingale $Y$. Since $\overline{\lambda}$ is a convex
function, $E\overline{\lambda}(Y_{r})$ is nondecreasing in $r$.
Therefore $I_{3,r}$ on the left-hand side of equality \ref{eq:temp-32}
is nondecreasing in $r$. 

3. By Definition \ref{Def. Harmoncic function}, the harmonic function
$u$ on $D$ is uniformly continuous on compact subsets of $D$. In
particular, $u$ is uniformly continuous on $\overline{D}_{0,r}$
with some modulus of continuity $\delta_{u,r}$ for each $r\in(0,1)$.
Consequently, the function $u(r\cdot)$ is uniformly continuous on
$\overline{D}_{0,1}$ with modulus of continuity $\delta_{u,r}$,
for each $r\in(0,1)$.

4. By Definition \ref{Def. Special convex function}, we have
\begin{eqnarray}
\overline{\lambda}(v) & \equiv & (e^{-|v|}-1+|v|)\label{eq:temp-524-1-1}
\end{eqnarray}
for each $v\in R$. Hence
\[
I_{3,r}\equiv\overline{\sigma}_{m,0,1}\overline{\lambda}(u(r\cdot))=\overline{\sigma}_{m,0,1}(e^{-|u(r\cdot)|}-1+|u(r\cdot)|)
\]
\[
=I_{2,r}-1+I_{1,r}\rightarrow\gamma\equiv b_{2}-1+b_{1}.
\]
as $r\uparrow1$. More precisely, for each $r\in(0,1)$ and $\varepsilon>0$,
we have 
\begin{equation}
0\leq\gamma-I_{3,r}=(b_{2}-I_{2,r})+(b_{1}-I_{1,r})<2^{-1}\varepsilon+2^{-1}\varepsilon=\varepsilon,\label{eq:temp-31}
\end{equation}
provided that 
\[
1-r<\delta_{3}(\varepsilon)\equiv\delta_{1}(2^{-1}\varepsilon)\wedge\delta_{2}(2^{-1}\varepsilon).
\]
Convergence relation \ref{eq:temp-37} and inequality \ref{eq:temp-48}
are proved. 

5. Define an increasing sequence $(r_{q})_{q=1,2,\cdots}$ in $(0,1)$
by
\begin{equation}
r_{q}\equiv1-\delta_{3}(\frac{1}{12}2^{-q}\exp(-3\cdot2^{q}b_{1}))\label{eq:temp-33}
\end{equation}
 for each\emph{ $q\geq1$}. Then $r_{q}\uparrow1$ as $q\rightarrow\infty$. 

6. Consider each $s,r\in(0,1)$ with $s\leq r$. Since the pair 
\[
(u(B_{\tau(s)}),u(B_{\tau(r)}))\equiv(u(B_{\tau(s)}),u(B_{\tau(r)}))
\]
of r.r.v.'s is a martingale according to Corollary \ref{Cor. Observations of a harmonic func on D of BM at successive exit times},
theorem 8.3.2 of \cite{Chan21} (the Bishop-Jensen inequality) is
applicable, to yield
\begin{equation}
0\leq E\overline{\lambda}(u(B_{\tau(r)}))-E\overline{\lambda}(u(B_{\tau(s)})).\label{eq:temp-124-1-2-1-1}
\end{equation}

7. Let $\varepsilon>0$ be arbitrary. Then, in view of equality \ref{eq:temp-32},
we can restate equality \ref{eq:temp-31} as
\begin{equation}
0\leq\gamma-E\overline{\lambda}(Y_{r}))<\varepsilon.\label{eq:temp-31-1-1-1}
\end{equation}
for each $r\in(1-\delta_{3}(\varepsilon),1)$. 

8. Now let $q\geq1$ be arbitrary. Write $\varepsilon_{q}\equiv2^{-q}$
for abbreviation. Then inequality \ref{eq:temp-33} implies
\begin{equation}
1-r_{q}<\delta_{3}(\frac{1}{6}\varepsilon_{q}\exp(-\frac{3}{\varepsilon_{q}}b_{1})),\label{eq:temp-33-1-1}
\end{equation}
or, equivalently,
\[
1-r_{q}<\widetilde{\delta}_{q}\equiv\delta_{3}(\widetilde{\varepsilon}_{q})
\]
where
\[
\widetilde{\varepsilon}_{q}\equiv\frac{1}{6}\varepsilon_{q}\exp(-\frac{3}{\varepsilon_{q}}b_{1}).
\]

9. Separately, recall from Definition \ref{Def. Modulus of continuity at 0 of standard normal CDF}
the integer $N_{2,q}$, with
\begin{equation}
2\Phi_{0,1}(\frac{2}{\sqrt{N(2,q)}})-1<2^{-q},\label{eq:temp-14-1-1}
\end{equation}
where $\Phi_{0,1}$ is the standard normal C.D.F.  Consider the first
exit time $\tau_{1}$. Fix an arbitrary $t_{q}>N_{2,q}.$ Then, applying
Assertion 4 of Theorem \ref{Thm.  Existence and continuity of first exit time},
where $x,\overline{r},r,\widetilde{r},k,t$ are replaced with $0,2^{-1},1,2,q,t_{q}$
respectively, we obtain 
\begin{equation}
P(\tau_{1}>t_{q})\leq2^{-q+1}.\label{eq:temp-43-1}
\end{equation}

10. Next, note that
\[
E|Y_{r(q)}|=E|u(B_{\tau(r(q))})|=\overline{\sigma}_{m,0,r(q)}|u|
\]
\begin{equation}
=\overline{\sigma}_{m,0,1}|u(r_{q}\cdot)|\equiv I_{1,r(q)}\leq b_{1},\label{eq:temp-45-1}
\end{equation}
where the inequality is by the first half of inequality \ref{eq:temp-31-2-1}.
At the same time, since
\[
1-r_{q}<\delta_{3}(\widetilde{\varepsilon}_{q})\leq\delta_{1}(2^{-1}\widetilde{\varepsilon}_{q})\leq\delta_{1}(\widetilde{\varepsilon}_{q}),
\]
we have, by inequality \ref{eq:temp-31-2-1}, 
\[
0\leq\gamma-E\overline{\lambda}(Y_{r(q)}))<\widetilde{\varepsilon}_{q},
\]
In view of the monotonicity of the sequence $(E\overline{\lambda}(Y_{r(q)}))_{q=1,2,\cdots}$,
this implies that
\[
0\leq E\overline{\lambda}(Y_{r(q+1)})-E\overline{\lambda}(Y_{r(q)})\leq\gamma-E\overline{\lambda}(Y_{r(q)})
\]
\[
<\widetilde{\varepsilon}_{q}\equiv\frac{1}{6}\varepsilon_{q}^{3}\exp(-\frac{3}{\varepsilon_{q}}b_{1})
\]
\begin{equation}
\leq\frac{1}{6}\varepsilon_{q}^{3}\exp(-\frac{3}{\varepsilon_{q}}(E|Y_{r(q)}|\vee E|Y_{r(q+1)}|)),\label{eq:temp-35-1}
\end{equation}
where the last inequality is thanks to inequality \ref{eq:temp-45-1}.

11. Continuing, write 
\[
\overline{\varepsilon}_{q}\equiv\delta_{u,r(q+1)}(\varepsilon_{q})>0.
\]
where we recall that $\delta_{u,r(q+1)}$ is a modulus of continuity
of the harmonic function $u$ on the compact subset $\overline{D}_{0,r(q+1)}$
of $D_{0,1}.$ Since the Brownian motion $B:[0,t_{q}]\times\Omega\rightarrow R^{m}$
is a.u. continuous on $[0,t_{q}]$, there exists a measurable set
$H_{q}\subset\Omega$ with 
\begin{equation}
P(H_{q})<\varepsilon_{q}\label{eq:temp-72}
\end{equation}
and some $\delta_{B,t(q)}(\overline{\varepsilon}_{q})>0$ such that
for each $\omega\in H_{q}^{c}$ and $v,s\in[0,t_{q}]$ with $|v-s|\leq\delta_{B,t(q)}(\overline{\varepsilon}_{q})$,
we have $\left\Vert B_{v}(\omega)-B_{s}(\omega)\right\Vert <\overline{\varepsilon}_{q}$. 

12. Now consider an arbitrary $n\geq1$ and arbitrary sequence $0\equiv v_{0}<v_{1}<\cdots<v_{n}\equiv t_{q}$
such that 
\begin{equation}
\bigvee_{k=1}^{n}|v_{k}-v_{k-1}|<\delta_{B,t(q)}(\overline{\varepsilon}_{q}).\label{eq:temp-36-1}
\end{equation}
Consider the martingale $Z_{0},\cdots,Z_{n},Z_{n+1}$ where 
\[
Z_{k}\equiv u(B_{\tau(r(q))\vee v(k)\wedge\tau(r(q+1))})
\]
for each $k=0,\cdots,n$, and where
\[
Z_{n+1}\equiv Y_{r(q+1)}\equiv u(B_{\tau(r(q+1))}).
\]
Take an arbitrary $\widehat{\varepsilon}_{q}\in(\varepsilon_{q},\varepsilon_{q-1})\equiv(2^{-q},2^{-q+1})$.
Then
\[
0\leq E\overline{\lambda}(Z_{n+1})-E\overline{\lambda}(Z_{0})
\]
\[
=E\overline{\lambda}(Y_{r(q+1)})-E\overline{\lambda}(Y_{r(q)})
\]
\[
\leq\frac{1}{6}\varepsilon_{q}^{3}\exp(-\frac{3}{\varepsilon_{q}}(E|Y_{r(q)}|\vee E|Y_{r(q+1)}|)
\]
\[
\leq\frac{1}{6}\widehat{\varepsilon}_{q}^{3}\exp(-\frac{3}{\widehat{\varepsilon}_{q}}(E|Y_{r(q)}|\vee E|Y_{r(q+1)}|)
\]
\[
=\frac{1}{6}\widehat{\varepsilon}_{q}^{3}\exp(-\frac{3}{\widehat{\varepsilon}_{q}}(E|Z_{0}|\vee E|Z_{n+1}|))),
\]
where the second inequality is from inequality \ref{eq:temp-35-1}.
Thus the conditions in Theorem \ref{Thm. Maximal inequality for a martingale}
are satisfied by the martingale $Z_{0},\cdots,Z_{n+1}$ and the constant
$\widehat{\varepsilon}_{q}$. Accordingly, inequality \ref{eq:temp-509}
of Theorem \ref{Thm. Maximal inequality for a martingale} is applicable
and yields
\begin{equation}
P(\bigvee_{k=0}^{n+1}|Z_{k}-Z_{0}|>\widehat{\varepsilon}_{q})<\widehat{\varepsilon}_{q}.\label{eq:temp-509-2-1-1}
\end{equation}
In other words,
\[
P(\bigvee_{k=0}^{n+1}|u(B_{\tau(r(q))\vee v(k)\wedge\tau(r(q+1))})-u(B_{\tau(r(q))})|>\widehat{\varepsilon}_{q})<\widehat{\varepsilon}_{q}.
\]
Hence
\begin{equation}
P(K_{q})<\widehat{\varepsilon}_{q}<2^{-q+1},\label{eq:temp-509-1-1-1-1}
\end{equation}
where
\begin{equation}
K_{q}^{c}\equiv(\bigvee_{k=0}^{n}|u(B_{\tau(r(q))\vee v(k)\wedge\tau(r(q+1))})-u(B_{\tau(r(q))})|\leq\widehat{\varepsilon}_{q}).\label{eq:temp-37-1}
\end{equation}

13. Now consider the measurable set 
\[
A_{q}\equiv G_{q}\cup H_{q}\cup K_{q}.
\]
Consider each $\omega\in A_{q}^{c}$. Then $\omega\in G_{q}^{c}\equiv(\tau_{1}\leq t_{q})\subset(\tau_{r(q+1)}\leq t_{q})$.
Let 
\[
s\in[\tau_{r(q)}(\omega),\tau_{r(q+1)}(\omega)]\subset[0,t_{q}]
\]
be arbitrary. Then there exists $k=0,\cdots,n$ such that 
\[
|\tau_{r(q)}\vee v_{k}\wedge\tau_{r(q+1)}(\omega)-s|\leq|v_{k}-s|\leq\bigvee_{i=1}^{n}|v_{i}-v_{i-1}|<\delta_{B,t(q)}(\overline{\varepsilon}_{q}).
\]
Hence, because $\omega\in H_{q}^{c}$, we have
\[
\left\Vert B_{\tau(r(q))\vee v(k)\wedge\tau(r(q+1))}(\omega)-B_{s}(\omega)\right\Vert <\overline{\varepsilon}_{q}\equiv\delta_{u,r(q+1)}(\varepsilon_{q}).
\]
Therefore, since $B_{\tau(r(q))\vee v(k)\wedge\tau(r(q+1))}(\omega),B_{s}(\omega)\in D_{0,r(q+1)},$
and since $\delta_{u,r(q+1)}$ is a modulus of continuity of the function
$u$ on $\overline{D}_{0,r(q+1)},$ we infer that
\begin{equation}
|u(B_{\tau(r(q))\vee v(k)\wedge\tau(r(q+1))}(\omega))-u(B_{s}(\omega))|<\varepsilon_{q}.\label{eq:temp-50-1}
\end{equation}
Moreover, since $\omega\in K_{q}^{c}$, it follows from equality \ref{eq:temp-37-1}
that 
\begin{equation}
|u(B_{\tau(r(q))\vee v(k)\wedge\tau(r(q+1))})-u(B_{\tau(r(q))}(\omega))|\leq\widehat{\varepsilon}_{q}<2^{-q+1}.\label{eq:temp-49-1}
\end{equation}
Combining inequalities \ref{eq:temp-50-1} and \ref{eq:temp-49-1},
we obtain
\begin{equation}
|u(B_{\tau(r(q))})-u(B_{s}(\omega))|<\varepsilon_{q}+\widehat{\varepsilon}_{q}<2^{-q}+2^{-q+1}<2^{-q+2},\label{eq:temp-76}
\end{equation}
where $s\in[\tau_{r(q)}(\omega),\tau_{r(q+1)}(\omega)]$ and $\omega\in A_{q}^{c}$
are arbitrary. Summing up,
\[
A_{q}^{c}\subset(\bigvee_{s\in[\tau(r(q)),\tau(r(q+1)]}|u(B_{\tau(r(q))})-u(B_{s})|\leq2^{-q+2}).
\]
and 
\[
P(A_{q})=P(G_{q})+P(H_{q})+P(K_{q})<2^{-q+1}+2^{-q}+2^{-q+1}<2^{-q+3},
\]
thanks to inequalities \ref{eq:temp-43-1}, \ref{eq:temp-72}, and
\ref{eq:temp-509-1-1-1-1}.

14. Now let $\overline{A}_{q}\equiv\bigcup_{k=q}^{\infty}A_{k}$.
Then $P(\overline{A}_{q})<2^{-q+4}$. Moreover 
\begin{equation}
\overline{A}_{q}^{c}\subset\bigcap_{k=q}^{\infty}(\bigvee_{s\in[\tau(r(k)),\tau(r(k+1)]}|u(B_{\tau(r(k))})-u(B_{s})|\leq2^{-k+1})\label{eq:temp-34-1-1}
\end{equation}
Hence
\begin{equation}
\overline{A}_{q}^{c}\subset\bigcap_{k=q}^{\infty}(|u(B_{\tau(r(k))})-u(B_{\tau(r(k+1))})|\leq2^{-k+1}).\label{eq:temp-35}
\end{equation}
Thus $u(B_{\tau(r(k))})\rightarrow V$ a.u. for some r.r.v. $V$.
Combining with \ref{eq:temp-34-1-1}, we obtain
\[
\overline{A}_{q}^{c}\subset\bigcap_{k=q}^{\infty}\{(|u(B_{\tau(r(k))})-V)|\leq2^{-k+2})\cap(\bigvee_{s\in[\tau(r(k)),\tau(r(k+1)]}|u(B_{\tau(r(k))})-u(B_{s})|\leq2^{-k+1})\}
\]
\[
\subset\bigcap_{k=q}^{\infty}(\bigvee_{s\in[\tau(r(k)),\tau(r(k+1)]}|V-u(B_{s})|\leq2^{-k+2}+2^{-k+1})
\]
\[
\subset\bigcap_{k=q}^{\infty}(\bigvee_{s\in[\tau(r(k)),\tau(r(k+1)]}|V-u(B_{s})|\leq2^{-q+3})
\]
\[
\subset(\bigvee_{s\in[\tau(r(q)),\tau(1))}|V-u(B_{s})|\leq2^{-q+3}),
\]
where $P(\overline{A}_{q})<2^{-q+4}$, where $q\geq1$ is arbitrary.
The theorem is proved.
\end{proof}
We conclude this article with a conjecture regarding the nontangential
limit of harmonic functions in the unit $m$-sphere.
\begin{defn}
\textbf{\label{Def. Nontangential convergence} (Convex hull, Stolz
domain, and Nontangential limit). }Let $A$ be an arbitrary subset
of $R^{m}$. The \emph{convex hull}\index{convex hull} of the set
$A$ is then defined as the set 
\[
\widehat{A}\equiv A^{\wedge}\equiv\{\alpha_{1}v_{1}+\cdots+\alpha_{k}v_{k}:k\geq1;v_{1}\cdots v_{k}\in A;\alpha_{1},\cdots,\alpha_{k}\geq0;\alpha_{1}+\cdots+\alpha_{k}=1\}
\]
of all convex combinations of points in $A$. Let $z\in\partial D_{0,1}$
and $a\in[0,1)$ be arbitrary. Then the subset 
\[
S_{z,a}\equiv(\{z\}\cup D_{0,a})^{\wedge}
\]
of $D_{0,1}$ is called a \index{Stolz domain.}\emph{Stolz domain.}

Let $g:D_{0,1}\rightarrow R$ be an arbitrary continuous function.
Then we say that the function $g$ has a \index{nontangential limit}\emph{nontangential
limit} $c$ at the point $z\in\partial D_{0,1}$ if, for each $a\in[0,1)$,
we have 
\[
\lim_{x\rightarrow z;x\in S(z,a)}g(x)=c.
\]
\end{defn}

$\square$

The classical nontangential-limit theorem can be reformulated as follows.
\begin{thm}
\textbf{\emph{\label{Thm. Classical-nontangential-limit theorem for Hardy spaces}
(Classical nontangential limit theorem for Hardy spaces)}}\textbf{.}
Let $u$ be an arbitrary member of the Hardy space $\mathbf{h}^{p}$.
Suppose condition\emph{ (i)} in Theorem \ref{Thm Brownian limit theorem for Hardy space}holds.
Then $u$ has a nontangential limit at each point $z$ in some measurable
subset with probability 1 in $\partial D_{0,1}$ relative to the uniform
distribution $\overline{\sigma}_{m,0,1}$.
\end{thm}

We conjecture that Theorem \ref{Thm. Classical-nontangential-limit theorem for Hardy spaces}
does not have a constructive proof. We conjecture however that if
condition (ii) in Theorem \ref{Thm Brownian limit theorem for Hardy space}
also holds then Theorem \ref{Thm. Classical-nontangential-limit theorem for Hardy spaces}
has a constructive proof.

\end{document}